\documentclass[12pt]{amsart}
\usepackage{mathrsfs}
\usepackage{eucal}
\usepackage{color}
\usepackage[all]{xy}
\usepackage{amssymb,amsmath}
\usepackage{setspace}

\usepackage{enumitem}

\date{May 6, 2015}

\calclayout
\newtheorem{dummy}{anything}[section]
\newtheorem{theorem}[dummy]{Theorem}
\newtheorem*{thma}{Theorem A}
\newtheorem*{thmb}{Theorem B}
\newtheorem*{thmc}{Theorem C}
\newtheorem*{thmd}{Theorem D}

\newtheorem{lemma}[dummy]{Lemma}
\newtheorem{proposition}[dummy]{Proposition}
\newtheorem{corollary}[dummy]{Corollary}
\theoremstyle{definition}%%Change Theoremstyle
\newtheorem{definition}[dummy]{Definition}

  \newtheorem{remark}[dummy]{Remark}
   
    \newtheorem{question}[dummy]{Question}
  \newtheorem*{acknowledgement}{Acknowledgements}

\newcommand
{\eqncount}{\setcounter{equation}{\value{dummy}}%
\addtocounter{dummy}{1}}
\newcommand{\cB}{\mathcal B}

\newcommand{\cN}{\mathscr N}

\newcommand{\cS}{\mathscr S}

\newcommand{\bZ}{\mathbb Z}
\newcommand{\bQ}{\mathbf Q}
\newcommand{\bP}{\mathbf P}

\newcommand{\PP}{\mathbb R \bP}

\newcommand{\bbR}{\mathbb R}

\newcommand{\bbZ}{\mathbb Z}

\newcommand{\ZZa}{{\bbZ[\bZ/2]}}

\DeclareMathOperator{\Coker}{coker}
\newcommand{\MK}{\mathbb M_{K}}

\newcommand{\cy}[1]{\bZ/{#1}}
\newcommand{\la}{\langle}
\newcommand{\ra}{\rangle}

\newcommand{\vv}{\, | \,}
\newcommand{\bd}{\partial}
\newcommand{\del}{\partial}
\newcommand{\dlra}[1]{\stackrel{#1}{\longrightarrow}}
\DeclareMathOperator{\sign}{sign}

\DeclareMathOperator{\id}{id}
\newcommand{\La}{\Lambda}
\newcommand{\Ze}{\bZ_{\varepsilon}}
\newcommand{\Za}{\bZ_{+}}
\newcommand{\Zb}{\bZ_{-}}
\DeclareMathOperator{\inn}{int}
\newcommand{\wM}{\widetilde{M}}
\newcommand{\wt}[1]{\widetilde{#1}}
\newcommand{\SW}{D}
\newcommand{\sw}{S}
\newcommand{\DD}{\wh D}
\newcommand{\CSS}{\cS_{\cat}}
\newcommand{\CTSS}{\cS^{\bf t}_{\cat}}
\newcommand{\CTNI}{\cN^{\bf t}_{\cat}}
\newcommand{\CNI}{\cN_{\cat}}

\DeclareMathOperator{\pt}{\mu_X}
\newcommand{\Nt}{\eta^{\bf t}}
\newcommand{\Nv}{\eta}
\DeclareMathOperator{\Id}{id}
\newcommand{\qsn}{QS^0}
\newcommand{\qsna}{(QS^0)}
\newcommand{\wh}{\widehat}
\newcommand{\Aut}{\mathcal{E}}
\newcommand{\Ad}{\textup{Ad}}

\newcommand{\MMK}{\mathcal{KM}_{4k{+}2}}
\newcommand{\MMKK}[1]{\mathcal{KM}_{#1}}

\newcommand{\xra}{\xrightarrow}

\newcommand{\an}[1]{\langle#1\rangle}
\newcommand{\WW }{W}
\newcommand{\VV }{V}
\newcommand{\MM}{M}
\newcommand{\NN}{N}
\newcommand{\XX}{X}
\newcommand{\YY}{Y}
\newcommand{\ZZ}{Z}
\newcommand{\cat}{{CAT}}
\newcommand{\pa}{p(\alpha)}
\newcommand{\pw}{p(w)}
\begin{document}
\title{Finite group actions on Kervaire manifolds}
\author{Diarmuid Crowley}
\address{Institute of Mathematics, University of Aberdeen, Aberdeen AB24 3UE, UK}
\email{dcrowley@abdn.ac.uk}
\author{Ian Hambleton}
\address{Department of Mathematics, McMaster University,  Hamilton, Ontario L8S 4K1, Canada}
\email{hambleton@mcmaster.ca }

\thanks{Research partially supported by NSERC Discovery Grant A4000. The authors would like  to thank the Max Planck Institut f\"ur Mathematik and the Hausdorff Research Institute for Mathematics in Bonn for hospitality and support.}

\begin{abstract} Let $\MK^{4k{+}2}$ be the Kervaire manifold:  a closed, piecewise linear (PL)  manifold with Kervaire invariant $1$ and the same homology as the product $S^{2k{+}1} \times S^{2k{+}1}$ of spheres. We show that a finite
group of odd order acts freely on $\MK^{4k{+}2}$ if and only if it acts freely on
$S^{2k{+}1} \times S^{2k{+}1}$. 
 If $\MK$ is smoothable, then each smooth structure on $\MK$ admits a free smooth involution.  If $k \neq 2^j-1$, then $\MK^{4k{+}2}$ does not admit any free TOP involutions.
Free ``exotic" (PL)  involutions are constructed on $\MK^{30}$, $\MK^{62}$, and
$\MK^{126}$. Each smooth structure on $\MK^{30}$ admits a free 
$\bZ/2 \times \bZ/2$  action.
\end{abstract}

\maketitle

\section{Introduction}
%%%%%%%%%%%%%%%%%%%%%%%%%%%%%%%%%%%%%%%%%%%%%%%%%%%

One of the main themes in geometric topology is the study of smooth manifolds and their piece-wise linear (PL) triangulations. Shortly after Milnor's discovery \cite{milnor5} of exotic smooth $7$-spheres, Kervaire \cite{kervaire2} constructed the first example (in dimension 10) of a PL-manifold  with no differentiable structure, and a new exotic smooth $9$-sphere $\Sigma^9$. 

The construction of Kervaire's 10-dimensional manifold was generalized to all dimensions of the form $m \equiv 2 \pmod 4$, via ``plumbing" (see \cite[\S 8]{hnk1}).
Let $P^{4k{+}2}$ denote the smooth, parallelizable manifold of dimension $4k{+}2$, $k \geq 0$, constructed by plumbing two copies of the the unit tangent disc bundle of $S^{2k{+}1}$. The  boundary $\Sigma^{4k+1} = \bd P^{4k{+}2}$ is a smooth homotopy sphere, now usually called the \emph{Kervaire sphere}. Since $\Sigma^{4k+1}$ is always PL-homeomorphic to the standard sphere  $S^{4k+1}$ (by Smale \cite{smale61}), one can cone off the boundary of $P^{4k{+}2}$ to obtain the \emph{Kervaire manifold}, denoted $\MK^{4k{+}2}$, with its canonical PL-structure. 

By construction, $\MK^{4k{+}2}$ is a closed, almost parallelizable, PL-manifold with the same homology as the product
$S^{2k{+}1}\times S^{2k{+}1}$ of spheres and it is simply-connected if $k > 0$.  
It admits a Wu structure $f_K$ with Arf invariant one (as defined by Kervaire \cite[\S 1]{kervaire2},  Kervaire-Milnor \cite[\S 8]{kervaire-milnor1}, and Browder \cite[\S 1]{browder1}). 
Moreover, $\MK^{4k{+}2}$ is \emph{minimal} with respect to these properties.
%since it is $2k$-connected and its Euler characteristic is zero.

\smallskip
In this paper, we consider symmetries of the Kervaire manifolds.
\begin{question} 
Does $\MK^{4k{+}2}$ admit any (PL) free orientation-preserving finite group actions~? If $\MK^{4k{+}2}$ is smoothable, does it admit any smooth free actions ?
\end{question}

If $\MK^{4k{+}2}$ is smoothable, then the Wu structure  $f_K$ is given by a framing, and the framed manifold $(\MK^{4k{+}2}, f_K)$ represents  an element in the stable stem  $\pi^S_{4k{+}2}$
by the Pontrjagin-Thom isomorphism. Conversely, any element in $ \pi^S_{4k{+}2}$ is represented by a smooth, closed, framed $(4k{+}2)$-manifold. However, Browder \cite{browder1} showed the Arf invariant of such manifolds is zero except in the special dimensions where $k = 2^j -1$, for some $j \geq 1$. Now the Arf invariant is preserved under framed cobordism, 
and we use the standard notation:
% it makes sense to define
  $$\theta_{j} \subset\pi^S_{2^{j+1}-2}$$
for  the subset of elements represented by smooth framed manifolds with Arf invariant one. In this notation, $\MK^{4k{+}2}$ is smoothable if and only if $\theta_{j+1}$ is non-empty, implying $k = 2^j-1$. 
  
% \begin{notation} The notation for  $\dim\theta_j$ is standard in homotopy theory, so for convenience of expostion we will reserve some special letters for dimensions.  Here is a list:
% \begin{enumerate}
% \item $m  = 2l=4k{+}2$, if $M^m$ is even-dimensional.
%\item $\dim\theta_j = 2^{j+1} -2 \equiv n$, when the index $j$ is understood.
%\item If $2k = \dim \theta_j$, then $4k{+}2 = 2^{j+2} - 2 = \dim\theta_{j+1} = 2\dim\theta_j +2$.
%\end{enumerate}
%  \end{notation}

It is now known that the Kervaire manifolds are smoothable (or equivalently that the Kervaire sphere is standard) in very few dimensions.   Kervaire \cite{kervaire2}  showed that $\MK^{10}$ does not admit any smooth structure, and 
Browder \cite{browder1} showed that $\Sigma^{4k+1}$ could only be diffeomorphic to $S^{4k+1}$ in the special dimensions 
$$4k{+}2=2^{j+2}-2 = \dim \theta_{j+1},$$ where $k = 2^j-1$. Note that when the Kervaire sphere is standard, the smooth structure resulting from attaching a $(4k{+}2)$-disk is not unique, since we may take connected sums with homotopy spheres in $\Theta_{4k{+}2}$, but all of the resulting smooth Kervaire manifolds are  stably parallelizable (by obstruction theory).
%, since the inertia group $I(M)=0$ for such a highly-connected, stably parallelizable manifold (see Kosinski \cite{kosinski1}).

Recently Hill, Hopkins and Ravenel \cite{hill-hopkins-ravenel1, hill-hopkins-ravenel2} have shown that 
$\Sigma^{4k+1}$ is not diffeomorphic to $S^{4k+1}$ if $k = 2^j-1$ and  $j \geq 6$. 
Earlier work of Barrett, Jones and Mahowald \cite{barratt-jones-mahowald1}, \cite{barratt-jones-mahowald2} showed that $\Sigma^{4k+1}$ is standard up to dimension 62 
($ j \leq 4$).  The 125-dimensional case is open.
%The case $j=5$ is open.

\smallskip
Here is a summary of the  results, first for involutions.

\begin{thma} Let $\MK^{4k{+}2}$ be a closed, oriented (PL) Kervaire manifold.
\begin{enumerate}
\item
If $\MK^{4k{+}2}$ is smoothable, then every smooth manifold $N$ with $N \cong_{PL} \MK^{4k{+}2}$
admits
% a framing $f$ such that $(N, f)$ admits 
 a smooth, free orientation-preserving involution.
\item Any smooth structure on $\MK^{30}$ admits a free, orientation-preserving smooth action of the group  $\bZ/2 \times \bZ/2$.
\item  If $4k{+}2 \neq 2^{j+2} -2$, then $\MK^{4k{+}2}$ does not admit any free (TOP) involutions. 
\end{enumerate}
\end{thma}

The first part of Theorem A will be proved in 
Theorem \ref{thm:Z2-action}, where the statement includes frame-preserving involutions,  and the second part in Theorem \ref{thm:dim30}. We remark that the last assertion of Theorem A is an immediate consequence of a result of Brumfiel, Madsen and Milgram \cite[Theorem 1.3]{brumfiel-madsen-milgram1}, which proves that $\MK^{4k{+}2}$ is an (unoriented) topological boundary if and only if $k = 2^{j}-1$. Since a manifold which admits a free involution bounds the unit orientation line bundle over its orbit space, even topological or orientation-reversing involutions are ruled out except in
 the ``Arf invariant dimensions"
$4k{+}2 = 2^{j+1}-2$.
In these cases, we have the following inductive construction.

\begin{thmb}  
%%%%%%%
Suppose that the set $\theta_j$  contains an element of order two, for some $j \geq 0$. Then  $\MK^{4k{+}2}$ admits  free, orientation-preserving (PL) involutions,  for $4k{+}2 = \dim \theta_{j+1}$.
\end{thmb}

For $j\leq 4$, when the Kervaire manifolds of dimension $\dim \theta_{j+1}$ are smoothable, Theorem A already provides a smooth, free, frame-preserving involutions. However, the  construction in Theorem B produces a wide variety of non-smoothable involutions in dimensions $30$ and above (see Theorem D and Theorem \ref{thm:thmb_detail}).
Moreover, the following result (for $j=5$)  gives a new symmetry of the Kervaire manifold in dimension $126$.

\begin{corollary}
 $\MK^{126}$ admits a free, orientation-preserving (PL) involution.
\end{corollary}

Note that $\MK^{126}$ is not currently known to be smoothable, but $\theta_5$ contains an element of order two
(see \cite{lin1}, \cite{kochman-mahowald1}),  and Theorem B applies. 
The situation for $\MK^{254}$ is at present unknown.  Moreover, Hill, Hopkins and Ravenel \cite{hill-hopkins-ravenel1} have shown that the sets $\theta_j$ are all empty, for $j \geq 7$, so the inductive construction of involutions via Theorem B cannot continue.

\smallskip
Here are some remaining problems:

\begin{question} \label{conj:non-action}
%%%%%%%%%%
Does the Kervaire manifold $\MK^{4k{+}2}$ admits a free, orientation-preserving   (PL)  involution if $4k +2 = \dim \theta_{j+1} \geq 254$~? 
Does $\bZ/2 \times \bZ/2$ act freely on some Kervaire manifold of dimension greater than 30~? 
\end{question}

\smallskip
In contrast, for odd order groups we have:

\begin{thmc} A finite group of odd order acts freely on  $\MK^{4k{+}2}$, preserving the orientation,  if and only if it acts freely on $S^{2k{+}1}\times S^{2k{+}1}$.
\end{thmc}

The proof of Theorem C
in Theorem \ref{thm:odd order}  is an application of the ``propagation" method of Cappell, Davis, L\"offler and Weinberger (see \cite{davis-loeffler1}, \cite{davis-weinberger2}).
This collection of actions includes some interesting finite groups, such as the extraspecial $p$-groups of rank 2 and exponent $p$ (see \cite{h-unlu1, h-unlu2}). We remark that the Kervaire manifolds $\MK^{4k{+}2}$ in the Arf invariant dimensions
$4k+2 = \dim \theta_j$ do not admit free, orientation-preserving (TOP) actions of non-abelian $p$-groups, for $p$ odd (these are ruled out by the cohomology ring structure: see \cite[Theorem A]{lewis2}).
 
 \medskip
We now discuss the proof of Theorem B.
In Theorem \ref{thm:decomp}, we show that the quotient  manifold 
$\MM:=\MK^{4k{+}2}/\la \tau\ra$ of any free smooth (or PL) involution on a Kervaire manifold
can be decomposed as a twisted double $\MM = \WW \cup_\phi \WW$. 
Here $\WW = D(\xi)$ is the disk bundle of a suitable PL-bundle of dimension $2k{+}1$ over $\PP^{2k{+}1}$, and $\phi\colon \VV \to \VV$ is a diffeomorphism (or PL-homeomorphism) of $\VV :=\bd \WW$.    The bundle $\xi$ is called the \emph{characteristic} bundle for  the involution, and $\xi$ is \emph{admissible} if $\pi^*(\xi) \cong \tau_{S^{2k{+}1}}$ under the standard projection $\pi\colon S^{2k{+}1} \to \PP^{2k{+}1}$ (see Proposition \ref{prop:stable_condition} for a stable recognition criterion). 
 
 In order to prove Theorem B, we construct such a twisted double decomposition, where $\phi$ is a PL-homeomorphism homotopic to an explicitly defined ``pinch map" homotopy equivalence $\pa \colon \VV \to \VV$ (see Theorem \ref{thm:M}). 
 The proof that the pinch map $\pa $ is homotopic to a PL-homeomorphism uses surgery theory as developed by Browder, Novikov, Sullivan and Wall (see \cite{wall-book}, \cite{browder72}). In this way, we construct examples with
 any admissible PL-bundle $\xi$ as the characteristic bundle for the involution
  (see Theorem \ref{thm:thmb_detail}).
 
In Section \ref{sec:tangential_surgery} we recall the main features of surgery theory for \emph{tangential} normal maps, following the work of Madsen, Taylor and Williams \cite[\S 2]{madsen-taylor-williams1}. In Section 
  \ref{sec:normal_invariants_of_pinch_maps}, we apply the theory of \cite{madsen-taylor-williams1} to obtain a 
  general formula for the tangential normal invariant of certain pinch maps (see Lemma \ref{lem:pinch}). 
 This formula may be of independent interest.
  
  The proof of Theorem B is completed in Section \ref{sec:induct}.
  %by showing that $h\colon \VV \to \VV$ is homotopic to a PL-homeomorphism $g\colon \VV\to \VV$. 
  The argument uses results of Brumfiel, Madsen and Milgram 
  \cite{brumfiel-madsen-milgram1} to analyze the image of the tangential normal invariant
  $\Nt(\pa) \in [\VV, SG]$ under the natural maps $SG \to G/O \to G/PL$. It follows that the Poincar\'e complex 
  $Z := \WW \cup_{\pa} \WW$ is homotopy equivalent to a PL-manifold $\MM$, and by our choice of characteristic bundle $\xi$ and pinch map $\pa $, we conclude that the universal covering $\wt{M}$ is PL-homeomorphic to $\MK^{4k{+}2}$ (see Theorem \ref{thm:homotopy_type} and
 Proposition \ref{prop:MK=tildeM}).
% Lemma \ref{lem:Kervaire_manifold}).
 
Finally, in Sections \ref{sec:spivak} and \ref{sec:smoothing_obstruction},  we show that some of the free (PL)  involutions on Kervaire manifolds constructed in Theorem B are ``exotic",  even if the characteristic bundle is a vector bundle (an action of \emph{linear type}).
 
\begin{thmd} 
%%%%%%%
There exist free orientation-preserving (PL) involutions of linear type on the Kervaire manifolds 
$\MK^{30}$, $\MK^{62}$ and $\MK^{126}$ which are not smoothable
\end{thmd}

 These actions on $\MK^{4k{+}2}$, for $4k{+}2 
  \in\{ 30, 62, 126\}$ are smoothable over the $(2k{+}1)$-skeleton (see Lemma \ref{lem:(n{+}1)_smooth}),
 but the stable PL-normal bundle 
 $\nu_\MM$ for the orbit space $\MM:=\MK^{4k{+}2}/\la \tau\ra$ does not admit a vector bundle structure 
 (see Corollary \ref{cor:not_smooth}). 
The proof of Theorem D relies on a result about the Spivak normal fibrations of
twisted doubles (see Proposition \ref{prop:Spivak}) which might have other applications.
 
\begin{acknowledgement} We would like to thank Bruce Williams, Jim Davis, Martin Olbermann, John Klein, Mark Behrens and Wolfgang Steimle for useful information. We would also like to thank the referee for helpful comments and suggestions.
\end{acknowledgement}
%
%%%%%%%%%%%%%%%%%%%%%%%%%%%%%%%%%%%%%%%%%%%%%%%%%%%%%%

\section{The proof of Theorem A} \label{sec:thma}
%%%%%%%%%%%%%%%%%%%%%%%%%%%%%%%%%%%%%%%%%%%%%%%%%%%%%%
The first part of Theorem A has been implicit in the literature since the 1970's (in particular, it does not use any of the recent progress concerning the $\theta_j$).  We first give a more detailed statement of the result.

\begin{theorem} \label{thm:Z2-action}
%%%%%%%%%%%%%%%%%%%
Suppose that $\MK^{4k{+}2}$ is smoothable. For any smooth, closed manifold 
$\NN \cong_{PL} \MK^{4k{+}2}$, and any  framing $(\NN , f)$ with Arf invariant one,  
$(\NN ,f)$ admits a smooth, free, frame-preserving involution.
\end{theorem}
The main step is due to E.~H.~Brown, Jr.~(based on work of N.~Ray and Kahn-Priddy; see also the remark \cite[p.\,664]{brumfiel-milgram1}).
\begin{theorem}[Brown {\cite{brown_edgar1}}] \label{thm:Brown}
%%%%%%%%%%%%%%%%%%%%%%%%
If $\alpha \in \pi^S_m$, $m >0$, then $\alpha$ can be represented by a smooth, closed, framed manifold $(\NN ,f)$, where
$\NN $ admits a smooth fixed-point free involution $\tau$ which preserves the framing $f$.
 If $\alpha \neq 0$ has $2$-primary order, then $(\NN ,f)$ and $\tau$ can be chosen so that $\NN $ is 
 $([m/2]-1)$-connected, and $(\NN ,f)/\la \tau\ra$ is framed cobordant to zero.
\end{theorem}

We will apply this result to the elements of $\theta_{j+1}$, where $m = 4k{+}2 = 2^{j+2}-2>2$.  Hence we assume that $4k{+}2 \in \{6, 14, 30, 62\}$ and
possibly that $4k{+}2 = 126$ if $\theta_6$ is non-empty.  
Let $\NN$ be a closed oriented smooth $2k$-connected $(4k{+}2)$-manifold.
Since $\pi_{2k{+}1}(BO) = \pi_{4k{+}2}(BO) = 0$, every such $\NN$ admits a framing $f$ of its
stable normal bundle and we let
\[ K(\NN, f) \in \bbZ/2 \]
denote the Kervaire invariant of $(\NN, f)$. 
For example, for $k = 0, 1, 3$,  there are framings $f_k$ of $S^{2k{+}1}$
such that $K(S^{2k{+}1} \times S^{2k{+}1}, f_k \times f_k) = 1$. 
On the other hand in dimensions $30, 62$ and possibly $126$, then
$K(\NN, f)$ is independent of $f$ \cite[\S 8]{kervaire-milnor1}.

Given an orientation-preserving diffeomorphism $g \colon \NN_0 \cong \NN_1$
and a framing $f$ of $N_1$, we obtain the induced framing $g^*(f)$ of $N_0$.  
Hence we may define the set,
\[
\MMK : = \{ (\NN, f)\vv K(\NN, f) = 1 \text{~and~}\chi(\NN) = 0 \},
\] 
of framed diffeomorphism classes of $2k$-connected closed smooth framed 
$(4k{+}2)$-manifolds with Kervaire invariant one and Euler characteristic zero. 
A result of Freedman \cite{freedman1976} and its proof), leads to the following
classification result for $\MMK$.

\begin{proposition}[{\cite[Theorem 1]{freedman1976}, 
\cite[Theorem 4]{kreck3}}] \label{prop:freedman}
For all $k > 0$, if $(\NN_0, f_0)$ and $(\NN_1, f_1)$ in $\MMK$ are framed cobordant,
then they are  framed diffeomorphic. 
\end{proposition}

\begin{proof}
In Freedman's notation, we take $(M, \xi)$ to be the trivial bundle over a point.
The proof \cite[Theorem 1]{freedman1976}, see also  \cite[Theorem 4]{kreck3},
shows that $(\NN_0, f_0)$  and $(\NN_1, f_1)$ are framed $h$-cobordant, and
hence framed diffeomorphic.
\end{proof}

It follows that the elements of $\MMK$ are in bijection with their framed cobordism classes 
in $\theta_{j+1}$ (see \cite[Theorem 6.6 and \S 8]{kervaire-milnor1} for surjectivity). 
The surface case ($k=0$) is left to the reader.

\begin{remark} \label{rem:Theta_and_KM}
Let $\Theta_{4k{+}2}$ denote the group of oriented $h$-cobordism classes of homotopy 
$(4k{+}2)$-spheres as defined in \cite{kervaire-milnor1}.    
By \cite[Lemma 4.5 and Lemma 8.4]{kervaire-milnor1} there is a short exact sequence
\eqncount
\begin{equation*} % \label{eq:km}
  0 \xra{} \Theta_{4k{+}2} \longrightarrow \Omega_{4k{+}2}^{\rm fr} \xra{~K~} \bbZ/2 \to 0
\end{equation*}
and hence $\Theta_{4k{+}2}$ acts freely and transitively on $K^{-1}(1) = \theta_{j+1}$. 
Since $\pi_{4k{+}2}(SO) = 0$, we may regard $\Theta_{4k{+}2}$ as the
group of framed diffeomorphism classes of framed homotopy spheres.
By the remarks above, we see that
$\Theta_{4k{+}2}$ also acts freely and transitively on the set $\MMK$
 via connected sum of framed manifolds. 
\end{remark} 

\begin{proof}[The proof of Theorem \ref{thm:Z2-action}]
%%%%%%%
If $\theta_j$ is non-empty, then by 
%the result of Brown 
the first sentence of Theorem \ref{thm:Brown}
there exists a smooth, closed, framed manifold $(\NN ,f)$, with Arf invariant one 
(and dimension $m=4k{+}2=2^{j+2}-2$), such that $\NN $ 
admits a smooth fixed-point free involution $t$ which preserves the framing $f$. 
By the second sentence of Theorem \ref{thm:Brown}, which is proven using
 equivariant framed surgery below the middle dimension, we may assume that $\pi_i(\NN ) = 0$ for $i <2k{+}1$.

The remaining part is contained in the second author's Ph.D thesis \cite{h0}.
Since $\NN $ is highly-connected, it follows that $H_{2k{+}1}(\NN ;\bbZ)$ is the direct sum (as a $\Lambda:=\bbZ[\bbZ/2]$-module) of a free $\Lambda$-module and two copies of the trivial $\Lambda$-module $\Za$.

 By \cite{h0} or \cite[Theorem 31]{hr1},  the $\bbZ[\bbZ/2]$-free summand splits off the $\bbZ/2$-equivariant intersection form of $\NN $, and supports a non-singular quadratic form 
 $$q\colon H_{2k{+}1}(\NN;\bbZ) \to Q_{-}(\bbZ/2^+) = \Lambda/\{\nu - \bar{\nu}\vv \nu \in \Lambda\} \cong \cy{2}\oplus \cy{2}$$
 refining the equivariant intersection form. The quadratic refinement $q$ is given by the framing at the identity element of  $\cy{2} = \{1, \tau\}$, and by the Browder-Livesay cohomology operation \cite[\S 4]{browder-livesay1} at the non-trivial element $\tau$.
 
 Hence we have an element of $L_{2l}(\bbZ[\bbZ/2],+)$, as discussed in \cite[\S 5]{wall-book}.  By \cite[\S 13A]{wall-book}, there are isomorphisms via the inclusion or projection map 
 $$L_{4k{+}2}(\bbZ[\bbZ/2],+) \cong L_{2}(\bbZ) \cong \bbZ/2.$$
 The Arf invariant of $(\NN ,f)$ is the sum of the Arf invariant of the form on the $\Lambda$-free part, and the Arf invariant of the hyperbolic form on 
 $
 \bbZ_{+}\oplus \bbZ_{+}$. We may choose the splitting of the equivariant intersection form so that the Arf invariant on the free part is zero. Then by equivariant framed surgery, the $\Lambda$-free summand can be removed. The new smooth framed manifold
  $(\NN ',f')$ has a smooth, free, frame-preserving involution, 
and $(\NN', f')$
is framed diffeomorphic to $(\NN, f)$ by Proposition \ref{prop:freedman}. 
\end{proof}

We now consider part (ii) of Theorem A.

\begin{theorem}\label{thm:dim30}
%%%%%%%%%%%%%%%%%
For any smooth, closed manifold $\NN \cong_{PL} \MK^{4k{+}2}$, of dimension $\leq 30$,  and any  framing $(\NN , f)$ with Arf invariant one,   $(\NN ,f)$ admits  a smooth, free,  frame-preserving $\bZ/2\times \bZ/2$ action.
\end{theorem}

\begin{proof} 
%%%%%%%
The idea is similar to the above: we use the fact that  the elements in $\theta_j$ are in the image of  the ``double transfer"
$$tr\colon \pi^S_{4k{+}2}(\PP^\infty \wedge \PP^\infty) \to \pi_{4k{+}2}^S$$
 for $j \leq 4$ (see
 Lin and Mahowald \cite{lin-mahowald1} for $\theta_4$). The double transfer is defined geometrically by taking the universal covering of a framed manifold
$(M,f)$ with a $2$-connected reference map $f\colon M \to \PP^\infty \times \PP^\infty)$. 
The argument is the same for each of the $\theta_j$, $j \leq 4$, but for $\dim \MK^{4k{+}2} < 30$ the Kervaire manifolds are products of spheres, with product framings, so a direct construction can be given.    Let $(M,f)$ be a smooth, closed, framed $30$-dimensional manifold, with 
 $$G:=\pi_1(M) = \bZ/2\times \bZ/2,$$
  such that its universal covering $(\wM, \tilde f)$ has Kervaire invariant one. By framed surgery below the middle dimension, we may assume that $\wM$ is $2k$-connected and has a free $G$-action preserving the framing.
\begin{enumerate}
\item The $\La$-module $H_{15}(\wM)$ is stably isomorphic to $L_0 \oplus L_1$, where $L_1$ is a free $\La$-module, and $L_0$ is an extension of $\Omega^{16}\bZ$ and its dual. We remark that the argument in  \cite[Prop.~2.4]{hk2} generalizes to $M$ since its universal covering is $2k$-connected.
\item The extension class for $L_0$ is the image $c_*[M] \in H_{30}(G; \bZ)$.
\item $c_*[M] \neq 0$ since $\Omega^{16}\bZ$ has $\bZ$-rank $> 1$ because $\bZ/2\times \bZ/2$ does not have periodic cohomology. 
\item For every class $u \in H^1(M;\cy 2)$, we have
 $u^{16} = 0$, but $u^{15}\neq 0$.
\item It follows that $0 \neq c_*[M] \in H_{15}(\cy 2) \otimes H_{15}(\cy 2) 
\subset H_{30}(G;\bZ)$. 
\item The fundamental class of $\PP^{15}\times \PP^{15}$  has the same image in $H_{30}(G;\bZ)$, hence $L_0 \cong \Za\oplus \Za$. 
\item The intersection form $\lambda_M$ is unimodular restricted to $L_0$, so it admits an orthogonal splitting $L_0 \perp L_1$. 
\end{enumerate}
 We can now do equvariant framed surgery to eliminate the free summand $L_1$, since the surgery obstruction group $L_{4k{+}2}(\bZ G) \cong \bZ/2$ is again detected by the ordinary Arf invariant (see \cite[Theorem 3.2.2]{wall-VI}). The resulting framed manifold $(N,f)$ has a smooth, free $G$-action preserving the framing, and 
 $N \cong_{PL} \MK^{4k+2}$ by Proposition \ref{prop:freedman} (PL-version).

\end{proof}
\begin{remark}
 Minami \cite{minami1} has proved that no order two element $x_5 \in \theta_5$ lies in the image of the double transfer, so this method of constructing $\bZ/2\times \bZ/2$ actions does not continue. 
 \end{remark}
 \begin{remark} \label{lem:MKK_count}
%%%%%%%%%%
Computations in homotopy theory provide the list:
$|\MMKK{2}| = |\MMKK{6}| = 1$, $|\MMKK{14}| = 2$, $|\MMKK{30}| = 3$ and $|\MMKK{62}| = 24$.
The values of $|\MMKK{4k{+}2}|$ for $4k{+}2  = 2, 6, 14$ and $30$ can be found in \cite[Table A3.3]{ravenel-greenbook}.
To determine $|\MMKK{62}|$ we use \cite{kochman-mahowald}.

For an $(\NN, f) \in \MMK$, the group $ H^{2k{+}1}(\NN; \pi_{2k{+}1}(SO)) \cong \bbZ^2$  acts freely 
and transitively on  the homotopy classes of framings of $N$ compatible with the orientation.  Hence there exist a large number of self-diffeomorphisms $g \colon \NN \cong \NN$ which act on the set of framings of $N$ (see 
\cite[Theorem 2]{kreck76} and \cite[Proposition 3.1]{crowley2011}).
\end{remark}

%%%%%%%%%%%%%%%%%%%%%%%%%%%%%%%%%%%%%%%%%%%%%%%%%%

\section{Free involutions on highly-connected manifolds}\label{sec:highlyconnected}
%%%%%%%%%%%%%%%%%%%%%%%%%%%%%%%%%%%%%%%%%%%%%%%%%%
Let $M^{2l}$ be a closed, oriented smooth or PL-manifold of dimension $2l\geq 4$, with fundamental group $\pi_1(M) =\cy 2$. In addition, we assume that $\pi_i(M) = 0$, for $1< i <l$, and consider the classification problem for such manifolds
This is equivalent to the study of free, orientation-preserving involutions on $(l{-}1)$-connected, $2l$-manifolds, by passing to the universal covering $\wM$ of $M$. We refer to \cite{h0, h1} and \cite{wells1,wells2} for earlier results on this problem, assuming $l \geq 3$, generalizing  the classification of $(l{-}1)$-connected $2l$-manifolds given by  Wall \cite{wall1962}.  Closed, oriented $4$-manifolds with fundamental group $\cy 2$ were classified by Hambleton and Kreck \cite{hk6}.

Let $\La =\ZZa$ denote the integral group ring, let $\bZ/2 = \la T\ra$, and let $\Za$ (respectively $\Zb$) denote the integers with $T$ acting as $+1$ (respectively $-1$). We will also write $\Ze$, with $\varepsilon = \pm 1$, for short.

\begin{lemma} Let $M^{2l}$ be a closed, oriented  PL-manifold of dimension $2l\geq 4$, with $\pi_1(M) =\cy 2$. If $\pi_i(M) = 0$, for $1< i <l$, then $\pi_l(M) \cong r\La\oplus \Ze \oplus \Ze$ for some $r\geq 0$,  with $\varepsilon = (-1)^{l+1}$. 
\end{lemma}
\begin{proof}
This is an easy consequence of the spectral sequence for the univeral covering $\wM \to M \to K(\cy 2,1)$.
\end{proof}
Next we recall the equivariant intersection form $\lambda_M\colon \pi_l(M) \times \pi_l(M) \to \bZ$, defined by counting interesections and self-intersections equivariantly in $\wM$ (see \cite[Chap.~5]{wall-book}). Then $\lambda_M$ is a unimodular $(-1)^l$-symmetic bilinear form, satisfying the properties (i) $\lambda_M(x, y) = \lambda_M(Tx, Ty)$, for all $x, y \in \pi_l(M)$, and (ii) $\lambda_M(x, Tx) \equiv 0 \pmod 2$, for all $x \in \pi_l(M)$.

In the rest of this section, we will consider only the special case $l = 2k{+}1$ relevant to the existence of free orientation-preserving smooth or PL-involutions on Kervaire manifolds. More precisely:

\begin{definition}\label{def:conditions}
Let $M^{4k{+}2}$ be a closed, oriented smooth or PL-manifold satisfying the following conditions:
\begin{enumerate}
\item $\pi_1(M) = \cy 2$,
\item $\pi_i(M) = 0$, for $1 < i <2k{+}1$, and
\item $H_{2k{+}1}(\wM;\bZ) \cong \pi_{2k{+}1}(M) \cong \Za \oplus \Za$.
\end{enumerate} 
\end{definition}  
We will give a geometric decomposition $M = \WW  \cup_\phi \WW $, based on the normal bundle $\xi$ of a \emph{characteristic} embedding $f\colon \PP^{2k{+}1} \to M$ (see Definition \ref{def:charembedding} and Theorem \ref{thm:decomp}).

For convenience, we will work now in the smooth category, but with obvious changes the discussion applies to the PL-category. Let $\cB = \{e_0, e_\infty\}$ denote a fixed symplectic base for $H_{2k{+}1}(\wM;\bZ)$, so that $\lambda_M(e_0,e_0) = \lambda_M(e_\infty, e_\infty) = 0$, and $\lambda_M(e_0, e_\infty) = 1$.
We first discuss the existence and uniqueness of embeddings $\PP^{2k{+}1} \subset M$.
\begin{definition} An embedding $f\colon \PP^{2k{+}1} \to M$ \emph{represents} 
$e_0 \in \ H_{2k{+}1}(\wM;\bZ)$ if, % pi_{2k{+}1}(\wM) 
\begin{enumerate}
\item $f_\#\colon \pi_1(\PP^{2k{+}1}) \to \pi_1(M)$ is an isomorphism,
\item $\tilde f_*([S^{2k{+}1}]) = e_0$ for some covering $\tilde f\colon S^{2k{+}1} \to \widetilde M$ of $f$.
\end{enumerate}
\end{definition}

\begin{proposition}  If $k\geq 1$, there is an embedding $f\colon \PP^{2k{+}1} \to M$ representing $e_0$, which is unique up to homotopy. If $k \geq 2$, the embedding is unique up to isotopy.
\end{proposition}

\begin{proof} Existence is proved in Wells \cite[Lemma 3]{wells1}. For uniqueness up to homotopy, we apply Olum \cite[Corollary 16.2]{olum1}, and uniqueness up to isotopy follows from Haefliger \cite{haefliger1}.
\end{proof}
\begin{corollary} If $k\geq 1$, the normal bundles in $M$ of any two embeddings of $\PP^{2k{+}1}$ representing $e_0$ are  isomorphic.
\end{corollary}

\begin{proof} For $k\geq 2$, the embeddings are isotopic so their normal bundles are isomorphic. If $k =1$, we have 
$f^*(\tau_M) \cong g^*(\tau_M)$, for any two homotopic embeddings. Therefore the normal bundles of $f$ and $g$ are stably isomorphic (see Fujii \cite[Theorem 2]{fujii1}). For $3$-plane bundles over $\PP^3$, stable isomorphism implies isomorphism (by Dold-Whitney \cite{dold-whitney1}).
\end{proof}

\begin{definition}\label{def:charembedding}
 A \emph{characteristic embedding} of $\PP^{2k{+}1}$ in $M$ is an embedding which represents $e_0 \in \cB \subset \pi_{2k{+}1}(M)$, where $\cB$ is a symplectic base for $\lambda_M$. The normal bundle to a characteristic embedding will be denoted $\xi = \xi(M)$, and called the \emph{characteristic bundle}. 
\end{definition}

The following lemma implies that every characteristic bundle has a section.

\begin{lemma} \label{lem:charsection}

Every orientable rank $2k{+}1$ vector bundle $\zeta$ over $\PP^{2k{+}1}$ admits a non-zero section.
\end{lemma}

\begin{proof}
Elementary obstruction theory shows that the Euler class of $\zeta$, $e(\zeta)$,
is the sole obstruction to the existence of a non-zero section. 
But $e(\zeta) \in H^{2k+1}(\PP^{2k{+}1}; \bZ) \cong \bZ$ has order two by \cite[Property 9.4]{milnor-stasheff1},
and so vanishes.  
\end{proof}

For the rest of this section, we fix a characteristic embedding $f\colon \PP^{2k{+}1} \to M$, and let $\WW  \subset M$ denote a small closed tubular 
neighbourhood of $f(\PP^{2k{+}1} )$ in $M$, with boundary $\VV  = \bd \WW $. 
Then $\WW $ is  diffeomorphic to $D(\xi)$, the total space of the $(2k{+}1)$-disk bundle associated to  the characteristic bundle, and $\VV $ is diffeomorphic to $S(\xi)$. 
Let $E = M - \inn \WW $ denote the complement of $\WW  \subset M$.

\begin{lemma} $E$ is diffeomorphic (PL-homeomorphic) to $\WW  \cong D(\xi)$.
\end{lemma}
\begin{proof}
By general position, we may isotope the embedding $f$ to obtain an embedding $g\colon \PP^{2k{+}1} \to M - \inn \WW $. 
This is possible because because the normal bundle $\xi$ has a non-zero section by Lemma \ref{lem:charsection}.
Then $g$ is unique up to isotopy, and we let $U \subset E= M -\inn \WW $ denote a small closed tubular neighbourhood of $g(\PP^{2k{+}1})$ in  $E$. It is easy to check that the region $E - \inn U$ is an $h$-cobordism between $\bd U$ and $\bd E = S(\xi)$. But $U\cong D(\xi)$, so 
it follows that $E$ is diffeomorphic to the total space of the characteristic $(2k{+}1)$-disk bundle $D(\xi)$ over $\PP^{2k{+}1}$.
\end{proof}

We summarize:

\begin{theorem}\label{thm:decomp}
 Suppose that $M^{4k{+}2}$ is a closed, oriented smooth (PL) manifold satisfying the condtions
 \textup{(\ref{def:conditions})},and  let $\xi(M)$ denote the normal bundle of a  characteristic embedding of $\PP^{2k{+}1}$ in $M$. 
Then there is a diffeomorphism (PL-homeomorphism) $\phi\colon S(\xi) \to S(\xi)$, such that 
$M \cong D(\xi) \cup_\phi D(\xi)$.
%$M \cong \WW  \cup_\phi \WW $.
\end{theorem}

This result will be our guide to constructing free involutions on the Kervaire manifolds.

\section{Twisted doubles and free involutions on Kervaire manifolds} 
\label{sec:FreeZ2actions}

We now consider the case when $M$ is a closed oriented PL-manifold, with  
fundamental group $\pi_1(M) = \cy 2$ and universal cover $\wM \cong_{PL} \MK^{4k{+}2}$ a Kervaire manifold. By Theorem A, this is only possible if $4k{+}2 = \dim \theta_{j+1} = 2\dim\theta_j +2$, for some $j\geq 0$. For convenience, we let $n= \dim \theta_j$ so that 
$\dim M = 2n{+}2$.
We recall a key feature of the plumbing description for the Kervaire manifolds. If $\nu$ denotes the normal bundle of an embedded 
$(2k{+}1)$-sphere in 
$\MK^{4k{+}2}$ which represents a primitive homology class, then 
$\nu \cong \tau_{S^{2k{+}1}}$ is isomorphic to the tangent bundle of the $(2k{+}1)$-sphere. Let $\pi\colon S^{2k{+}1} \to \PP^{2k{+}1}$ denote the 2-fold covering projection.

By Theorem \ref{thm:decomp}, to construct a suitable orbit manifold $M := \MK^{4k{+}2}/\la \tau\ra$, we need to find  the following:

\begin{enumerate}
\item A $(2k{+}1)$-dimensional (PL) bundle $\xi$ over $\PP^{2k{+}1}$, such that $\pi^*(\xi) \cong \tau_{S^{2k{+}1}}$. 
 \item %An orientation-reversing 
A PL-homeomorphism $g\colon S(\xi) \to S(\xi)$, so that the manifold  $M_g: = \WW \cup_g \WW $, with $\WW  = D(\xi)$, will have 
 universal covering $\wM_g \cong \MK^{4k{+}2}$.
\end{enumerate}

Note that for $l \neq 1, 3, 7$, the tangent bundle $\tau_{S^l}$ is the unique non-trivial $l$-plane bundle over $S^l$ which is stably trivial. 

The first requirement can clearly be met by taking $\xi = \tau_{\PP^{2k{+}1}}$. In the Arf invariant dimensions, there is another possibility:
\begin{theorem}[Brown \cite{brown_edgar2}]  Let $\nu$ denote the normal bundle of a smooth immersion of $\PP^l$ in $\bbR^{2l}$. If $l \neq 1, 3, 7$ and $l$ is odd, then $\pi^*(\nu)$ is isomorphic to $\tau_{S^l}$ if and only if $l = 2^j-1$, for some $j >3$.
\end{theorem}

This choice fits well with the construction of smooth frame-preserving free involutions in the cases where $\MK^{4k{+}2}$ is smoothable, since then $\WW  = D(\xi)$ will be parallelizable. In general, we can take any PL-bundle $\xi$ of dimension $2k{+}1$, with the required property for $\pi^*(\xi)$. 

\begin{definition}\label{def:admissible}
A PL-bundle $\xi$ of dimension $2k{+}1$ over $\PP^{2k{+}1}$ is called an \emph{admissible}  bundle if $\pi^*(\xi) \cong \tau_{S^{2k{+}1}}$. If $M$ has characteristic bundle $\xi = \xi(M)$, then we will say that $\wM$ has  an  \emph{involution of type $\xi$}.
\end{definition}

Here is a \emph{stable} characterization of admissible  bundles.
Let $i\colon \PP^{2k{+}1} \to \PP^{2k{+}2}$ denote the standard inclusion. 

\begin{proposition}\label{prop:stable_condition} 
Let $\xi$ be a PL-bundle of dimension $2k{+}1$, for $k \geq 4$, with $\pi^*(\xi)$ stably trivial, and let  $\gamma \in KPL(\PP^{2k{+}1})$ denote the stable equivalence class of $\xi$. Then
 $\pi^*(\xi) \cong \tau_{S^{2k{+}1}}$ if and only if
  there exists $\hat \gamma \in KPL(\PP^{2k{+}2})$, such that $i^*(\hat\gamma) = \gamma$ and
$w_{2k{+}2}(\hat\gamma) \neq 0$.
\end{proposition}

We first recall some facts about PL-bundles and discuss the stable conditions. 
\begin{enumerate}
\addtolength{\itemsep}{0.2\baselineskip}
\item By assumption, $\pi^*(\xi)$ is stably trivial and hence $\pi^*(\gamma)$ is also stably trivial. It follows from the
 cofibration sequence
$$ KPL(S^{2k{+}2}) \to  KPL(\PP^{2k{+}2}) \xrightarrow{i^*} KPL(\PP^{2k{+}1}) \xrightarrow{\pi^*} KPL(S^{2k{+}1}),$$
 that  $i^*(\hat\gamma) = \gamma$, for some $\hat\gamma \in KPL(\PP^{2k{+}2})$.
For vector bundles, $KO(\PP^{2k{+}1})$ is additively generated by the canonical line bundle $\eta \searrow \PP^{2k{+}1}$ (see Fujii \cite{fujii1}), 
 so this condition is automatic.

\item Any stable bundle $\hat\gamma$ over $\PP^{2k{+}2}$ admits an unstable reduction to a $(2k{+}2)$-dimensional bundle $\xi_0$ (see Haefliger and Wall \cite{haefliger-wall1}).  
Recall that $w_{2k{+}2}(\hat\xi_0) = w_{2k{+}2}(\hat\gamma)$ is the mod 2 reduction of the twisted Euler class $$e(\hat\xi_0) \in H^{2k{+}2}(\PP^{2k{+}2}; \Zb).$$ 
 By obstruction theory, $w_{2k{+}2}(\hat\gamma) = 0$ if and only if there exists a $(2k{+}1)$-dimensional reduction $\hat\xi$ of $\hat\gamma$. 
 
 \item The characteristic class $w_{2k{+}2}(\hat\gamma) \in H^{2k{+}2}(\PP^{2k{+}2};\cy 2)$ is independent of the choice of extension $\hat\gamma$ with $i^*(\hat\gamma) = \gamma$.  By Adams \cite{adams1960}, the class $w_{2k{+}2}(\zeta) \equiv 0 \pmod 2$ for  a $(2k{+}2)$-bundle $\zeta$ over $S^{2k{+}2}$, since $k \geq 4$.

\item Since $k \geq 4$, the tangent bundle $\tau_{S^{2k{+}1}}$ is the unique non-trivial vector bundle of dimension $2k{+}1$ over $S^{2k{+}1}$ which is stably trivial. 
  For PL-bundles, we use the results of  Burghelea and Lashof
\cite[II, \S 5]{burghelea-lashof1}. By stability \cite[Proposition 5.6]{burghelea-lashof1}, we may use PL-bundles instead of PL-block bundles. Then by
\cite[Theorem 5.1']{burghelea-lashof1}, the same uniqueness statement holds for $\tau_{S^{2k{+}1}}$ as a PL-bundle. Hence, the stably trivial bundle $\pi^*(\xi)$ is either trivial or $\pi^*(\xi) \cong \tau_{S^{2k{+}1}}$. 

\item Note also that $\pi^*(\xi) \cong \pi^*(\xi')$ for any two $(2k{+}1)$-dimensional reductions $\xi$ and $\xi'$ of $\gamma$, since $\tau_{S^{2k{+}1}}$ has order two. 
Note that $\xi$ and $\xi'$ differ only on the top $(2k{+}1)$-cell, and applying $\pi^*$ multiplies the bundle by two. 
\end{enumerate}

\begin{proof}[The proof of Proposition \ref{prop:stable_condition}]
Suppose that $\xi$ is some PL-bundle of dimension $2k{+}1$, $k \geq 4$,  with stable class $\gamma \in KPL(\PP^{2k{+}1})$, and  $\pi^*(\xi)$ stably trivial.
   If $\pi^*(\xi)$ is actually the trivial bundle, then 
  the cofibration sequence 
 $$  [\PP^{2k{+}2}, BPL_{2k{+}1}] \xrightarrow{i^*} [\PP^{2k{+}1}, BPL_{2k{+}1}] \xrightarrow{\pi^*} [S^{2k{+}1}, BPL_{2k{+}1}]$$
implies that $i^*(\hat\xi) = \xi$ for some $(2k{+}1)$-bundle over $\PP^{2k{+}2}$. 
Let $\hat\gamma$ denote the stable class of $\hat\xi$, so $i^*(\hat\gamma) = \gamma$. Since $\hat\xi$ is a $(2k{+}1)$-dimensional reduction of $\hat\gamma$, we see that $w_{2k{+}2}(\hat\gamma) = 0$. 

Conversely, if  $\pi^*({\xi})$  is non-trivial  then $\pi^*({\xi})\cong \tau_{S^{2k{+}1}}$. Let $\hat\gamma$ be a stable PL-bundle over $\PP^{2k{+}2}$ such that 
$i^*(\hat\gamma)= \gamma$. Then $w_{2k{+}2}(\hat\gamma) = 0$ would imply that $\hat\gamma$ has a $(2k{+}1)$-dimensional reduction $\hat\xi$, and hence $\xi' = i^*(\hat\xi)$ would be a $(2k{+}1)$-dimensional reduction of $\gamma$. But 
$\pi^*(\xi') = \pi^*(i^*(\hat\xi))$ is trivial, and this is a contradiction since 
$\pi^*(\xi) \cong \pi^*(\xi')$.
 \end{proof}
 
As mentioned above, the group $\wt{KO}(\PP^{2k{+}1})$ %\cong \bZ/d$ 
is cyclic with generator the reduced class 
of the non-trivial line bundle $\eta$ over $\PP^{2k{+}1}$. 
 
\begin{corollary} \label{cor:admisable_smooth_bundles}
%%%%%%%%
A $(2k{+}1)$-dimensional vector bundle $\xi$ over $\PP^{2k{+}1}$ is admissible if and only 
if its stable class $\gamma = m\cdot\eta$ satisfies 
$\binom{m}{2k{+}2} \equiv 1 ~\textup{mod}~2 $. 
\end{corollary}

\begin{proof}
%%%%%%%
By the Cartan formula, the total Stiefel-Whitney class of $m \cdot \eta$ is $(1+x)^m$
where $x \in H^1(\PP^{2k{+}1}; \bZ/2)$ is a generator.   Now apply Proposition \ref{prop:stable_condition}. 
\end{proof}

The main step in the proof of Theorem B is based on the following important result from homotopy theory. Let 
$[\iota_{n+1}, \iota_{n+1}] \in \pi_{2n+1}(S^{n+1})$ denote the Whitehead square.

\begin{theorem}[{Barratt-Jones-Mahowald \cite[Cor.~3.2]{barratt-jones-mahowald1}}] \label{thm:B-J-M}
%%%%%%%%
Let $n = 2^{j+1} -2$. 
There exists an element of order two in $\theta_j$ with Kervaire invariant one if and only if 
$[\iota_{n+1}, \iota_{n+1}] = 2\alpha$, for some $\alpha \in  \pi_{2n+1}(S^{n+1})$.
%\remdc{I have seen this result attributed to Brown-Peterson:
%bibitem[B-P]{B-P} E.~H.~Brown and F.~P.~Peterson, {\rm Whitehead products and cohomology operations.}, 
%Quart. J. Math. Oxford Ser. (2) 15 (1964) 116--120.}
\end{theorem}

A map $\alpha$ given by this result is will be said to \emph{halve the Whitehead square}. If $\VV  = S(\xi)$ for 
some admissible bundle $\xi$, then there is a section $s\colon \PP^{2k{+}1} \to \VV $ arising from a non-zero 
section of $\xi$. Note that in the notation $n= \dim\theta_j$, we have $4k{+}2 = 2n{+}2$.

\begin{definition} \label{def:p_alpha}
%%%%%%%%%
 Suppose that $[\iota_{n+1}, \iota_{n+1}] = 2\alpha$, for some 
 $\alpha \in  \pi_{2n+1}(S^{n+1})$, and let $\VV  = S(\xi)$. 
By Lemma \ref{lem:charsection}, the bundle $V \to \PP^{2k{+}1}$ admits a section 
$s \colon \PP^{2k{+1}} \to V$, and we
define the pinch map
$\pa \colon \VV  \to \VV $ as the composite
$$\pa : \xymatrix{\VV  \ar[r] & \VV  \vee S^{2n+1} \ar[r]^{\id \vee \alpha}&
\VV \vee S^{n+1}\ar[r]^(0.45){\id \vee \pi}& \VV  \vee \PP^{n+1}\ar[r]^(0.65){\id  \vee s}& \VV. }$$
\end{definition}
In the Sections \ref{sec:induct} and \ref{sec:smoothing_obstruction}, we will analyze the normal invariants 
of the pinch maps $\pa$ constructed by halving the Whitehead square.  For future use, we prove that $\pa$ preserves $\nu_\VV$, the stable normal bundle of $\VV$.

\begin{lemma} \label{lem:pa_is_tangential}
%%%%%%%
$\pa^*(\nu_\VV) \cong \nu_\VV$. 
\end{lemma}

\begin{proof}
%%%%%%%
It is enough to show that $(s \circ \pi \circ \alpha)^*(\nu_\VV) = \alpha^*(\pi^*(s^*(\nu_\VV)))$ is trivial.  
Now $\VV  = S(\xi)$ is the total space of the sphere bundle of $\xi$, and therefore
\[ \nu_{\VV } \cong \pi_\xi^*(\nu_{\PP^{n+1}}) \oplus \pi_\xi^*(-\gamma), \]
where $\pi_\xi \colon \VV  \to \PP^{n+1}$ is the bundle projection, 
$\gamma$ is the stable bundle defined by $\xi$ and $-\gamma$ its stable inverse.  
Since $s \circ \pi_\xi = \Id_{\PP^{n+1}}$, 
\[  s^*(\nu_{\VV})= \nu_{\PP^{n+1}} \oplus (-\gamma), \]
where $\nu_{\PP^{n+1}}$ is the stable normal bundle of $\PP^{n+1}$.   Now $\nu_{\PP^{n+1}} \cong (n{+}2) \cdot \eta$ by \cite[Theorem 4.5]{milnor-stasheff1}, and it follows that $\pi^*(\nu_{\PP^{n+1}})$ is trivial.  By the definition of admissibility,  
$\pi^*(\gamma)$ is trivial.  Hence $\pi^*(s^*(\nu_{\VV}))$ is trivial, proving the lemma. 
\end{proof}

%%%%%%
\section{Pinch maps and the Kervaire manifold} \label{sec:pinch}
%%%%%%%%%%%%%%%%%%%%%%%%%%%%%%%%%%%%%%%%

We begin with the definition of a pinch map.  Let $\XX $ be a closed $m$-manifold and let
$ x \in \pi_m(\XX ) $
be a homotopy class of degree zero.  The pinch map on $x$ is a self-homotopy equivalence $p(x)$ defined as the composite
 \[ p(x): \xymatrix{\XX  \ar[r] & \XX  \vee S^{m}  \ar[r]^(0.65){\id \vee x} & \XX }. \]
With this notation, the map of Definition \ref{def:p_alpha} is $\pa  = p(s \circ \pi \circ \alpha)$.  In this section we
 show that the double covering of the pinch map $p(\alpha)$ can be used to construct the homotopy 
type of the Kervaire manifold $\MK^{4k{+}2}$.

\begin{theorem}\label{thm:homotopy_type} 
Let $\WW = D(\xi)$, for $\xi$ an admissible PL-bundle. If  $[\iota_{n+1}, \iota_{n+1}] = 2\alpha$, for some  $\alpha \in  \pi_{2n+1}(S^{n+1})$, 
then the Poincar\'e complex $\ZZ: = \WW\cup_{\pa } \WW$ constructed from the pinch map $\pa $ has universal covering 
$\wt{\ZZ} \simeq \MK^{4k{+}2}$.
\end{theorem}

From its construction, it is clear that the homotopy type of the Kervaire manifold is given by attaching a $(4k{+}2)$-cell to a
wedge of two $S^{2k{+}1}$-spheres:
\[   \MK^{4k{+}2} \simeq (S^{2k{+}1}_0 \vee S^{2k{+}1}_1) \cup_\varphi D^{4k{+}2}.   \]
The homotopy class of the attaching map $\varphi \colon S^{4k+1} \to S^{2k{+}1}_0 \vee S^{2k{+}1}_1$ 
is well known to experts,
but we did not find an explicit statement in the literature. 

\begin{lemma}[c.f.~{\cite[Lemma 8]{wall1962},  \cite[Lemma 8.10]{baues}}] \label{lem:KM}
Let $i_0, i_1 \colon S^{2k{+}1} \to S^{2k{+}1}_0 \vee S^{2k{+}1}_1$ be the inclusion maps of the $(2k{+}1)$-sphere
onto the indicated components of the wedge,
$[i_0, i_1]$ their Whitehead product and $w \in \pi_{4k+1}(S^{2k{+}1})$ the Whitehead square.  
Then
%\eqncount
\begin{equation*} %\label{eq:KM}
[\varphi] = [i_0, i_1] + i_{0}(w) + i_{1}(w) \in \pi_{4k+1}(S^{2k{+}1}_0 \vee S^{2k{+}1}_1) .
\end{equation*}
\end{lemma}

\begin{proof}
By the Hilton-Milnor Theorem, \cite[XI, \S 6]{whitehead-book}, we have 
$$ [\varphi] = r[i_0, i_1] +  i_{0}(y_0) + i_{1}(y_1),$$ 
where $y_i \in \pi_{4k+1}(S^{2k+1}_i)$.   The non-singularity of the cup product pairing
on $H^{2k+1}(\MK^{4k+2}; \bZ)$ ensures that $r = 1$.
To determine the homotopy classes classes $x_i$, we look at the collapse map 
$c \colon \MK^{4k+2} \to T(\nu_i)$ where $\nu_i$ is the normal bundle of $S^{2k+1}_i \subset \MK^{4k+2}$
and $T(\nu_i)$ is the Thom space of $\nu_i$.
From the construction of $\MK^{4k+2}$, we see that $\nu_i = \tau_{S^{2k+1}}$, 
and so 
$$ T(\nu_i) = S_i^{2k+1} \cup_{x_i} D^{4k+2},$$
where by \cite[Lemma 1]{milnorJ}  
$x_i = J(\tau_{S^{2k+1}})$ and $J \colon \pi_{2k}(SO(2k{+}1)) \to \pi_{4k+1}(S^{2k+1})$ is the 
$J$-homomorphism of \cite[XI Theorem 4.1]{whitehead-book}.  
Now since since the collapse map $c$ has degree one, $y_i = x_i$
and by \cite[(1.2)]{james-whitehead1}, $x_i = J(\tau_{S^{2k+1}}) = w$.
Hence $y_i = w$ for $i = 0, 1$, which completes the proof.
%can be determined from \cite[Lemma 8.3]{kervaire-milnor1} and is given by 
\end{proof}

Recall that 
$\WW $ is the total space  of a $D^{2k{+}1}$-bundle $\xi$ over $\PP^{2k{+}1}$ whose universal cover 
$\wt \WW $ is PL-homeomorphic to
the unit tangent disc bundle of $S^{2k{+}1}$.  The boundary $\wt\VV=\bd \wt \WW $ is thus the unit
tangent sphere bundle of $S^{2k{+}1}$.  There is a section $\wt s \colon S^{2k{+}1} \to \wt \VV $ covering
the section $s \colon \PP^{2k{+}1} \to \VV$.  We define the pinch map $\pw  := p(\wt s \circ w)$ to be the 
self-homotopy equivalence,
 $$\pw : \xymatrix{\wt \VV  \ar[r] & \wt \VV  \vee S^{4k+1} \ar[r]^{\id \vee w}& \wt \VV  \vee S^{2k{+}1} 
 \ar[r]^(0.625){\id  \vee \,\wt s}& \wt \VV, }$$
and the Poincar\'{e} complex,
$$ \ZZ_w : = \wt{\WW} \cup_{\pw } \wt{\WW},$$
obtained by gluing two copies of $\wt\WW$ together using $\pw $.

\begin{lemma} \label{lem:p_w}
%%%%%%%%
There is a homotopy equivalence $\MK^{4k{+}2} \simeq \ZZ_w$.
\end{lemma}

\begin{proof}  
To identify the homotopy type of $\ZZ_w$, we compare it to the trivial double 
\[ \ZZ_{\id} : = \wt \WW \cup_{\id} \wt \WW \cong S^{2k{+}1} \times S^{2k{+}1} .\]
Let $S^{2k{+}1}_0$ denote the zero section of one copy of $\wt \WW  \subset \ZZ_{\id}$ and let $S^{2k{+}1}_1$
denote a copy of the transverse sphere constructed from two fibre $(2k{+}1)$-disks in the copies of 
the bundle $\wt \WW \to S^{2k{+}1}$.   Applying \cite[Lemma 8.3]{kervaire-milnor1} we deduce that there 
is a homotopy equvialence
\[ \ZZ_{\id} \simeq (S^{2k{+}1}_0 \vee S^{2k{+}1}_1)\cup_{\varphi(\id)} D^{4k{+}2}, \]
where $[\varphi(\id)]  = [i_0, i_1]+ i_1(w)$.  
Since $p(w) \colon \wt \VV  \cong \wt \VV $ is a pinch map on $\wt s \circ w$, it follows that
there is a homotopy equivalence 
\[ \ZZ_w \simeq (S^{2k{+}1}_0 \vee S^{2k{+}1}_1)\cup_{\varphi(w)} D^{4k{+}2}, \]
where $[\varphi(w)] = [\varphi(\id)] + i_0(w)$.  Thus $\varphi(w) = [i_0, i_1] + i_0(w) + i_1(w)$ and
so by Lemma \ref{lem:KM}, $\ZZ_w$ is homotopy equivalent to $\MK^{4k{+}2}$.  
\end{proof}

\begin{lemma} \label{lem:ptop}
%%%%%%%
The homotopy equivalence $\pa  \colon \VV \simeq \VV$ lifts to $\pw \colon \wt \VV \simeq \wt \VV$.
\end{lemma}
\begin{proof}
For an oriented double cover $\wt \XX \to \XX$ with non-identity deck transformation $\tau$, it is a simple matter to check
that the double cover $\wt p(x) \colon \wt \XX \simeq \wt \XX$ of a pinch map $p(x) \colon \XX \simeq \XX$ on $x$,  satisfies
\[  \wt p(x) = p(\wt x + \tau \wt x), \]
where $\wt x \in \pi_m(\wt \XX) \cong \pi_m(\XX)$ is a lift of $x$.  The lemma follows since 
$\pa  = p(s \circ \pi \circ \alpha)$ pinches along $s(\PP^{n+1})\subset \VV$ and
the deck transformation of the covering $\pi \colon S^{n+1} \to \PP^{n+1}$ is 
homotopic to the identity and so acts trivially on homotopy groups.  Thus
\[ \wt \pa  = \wt p(s \circ \pi \circ \alpha) = p(\wt s \circ \alpha + \wt s \circ \alpha) = p(\wt s \circ (2 \alpha))
 = p(\wt s \circ w) = \pw .\]
\end{proof} 
\begin{proof}[The proof of Theorem  \ref{thm:homotopy_type}] We note that
Lemma \ref{lem:p_w} and Lemma \ref{lem:ptop} imply that $\ZZ_w \simeq \wt\ZZ$,
which completes the proof. 
\end{proof}
%
%%%%%%%%%%%%%%%%%%%%%%%%%%%%%%%%%%%%%%%%%%%%%%%%%

\section{Tangential  surgery} \label{sec:tangential_surgery}
%%%%%%%%%%%%%%%%%%%%%%%%%%%%%%%%%%%%%%%%%%%%%%

In this section we recall the
tangential surgery exact sequence and in particular the definition of the normal invariant of
a tangential degree one normal map.  
Our discussion follows \cite[\S2, \S4]{madsen-taylor-williams1} closely, 
however our setting is for closed manifolds, whereas Madsen, Taylor and Williams considered manifolds
with boundary.  

Let $\XX $ be a closed $m$-dimensional manifold, either smooth of PL, with stable normal bundle $\nu_\XX$ 
of rank $k>\!\!>m$.   
The $\cat$ tangential structure set of $\XX $, 
\[  \CTSS(\XX ) : = \{ (\MM , f, b) \vv  f \colon \MM \to \XX, b\colon  \nu_{\MM} \to \nu_\XX \} / \simeq, \]
consists of equivalences classes of triples $(\MM , f, b)$ where $f \colon \MM  \to \XX $ is a homotopy equivalence
and $b \colon \nu_{\MM} \to \nu_\XX $ is a map of stable bundles.
Two structures $(\MM _0, f_0, b_0)$ and $(\MM _1, f_1, b_1)$ are equivalent if there is an $s$-cobordism 
$(U; \MM _0, \MM _1, F, B)$ where $F \colon U \to \XX $ is a simple homotopy equivalence, 
$F \colon \nu_U \to \nu_\XX $ is a bundle map and these data restrict to 
$(\MM _0, g_0, b_0)$ and $(\MM _1, g_1, b_1)$ at the boundary of $U$.

Let $\pi = \pi_1(\XX )$.  The tangential surgery exact sequence for $\XX $ finishes with the following four terms
\eqncount
\begin{equation} \label{eq:ses}
L_{m+1}(\bZ\pi) \dlra{\rho}  \CTSS(\XX ) \dlra{\omega} \CTNI(\XX ) \dlra{\sigma} L_m(\bZ\pi), 
\end{equation}
where $L_\ast(\bZ\pi)$ are the surgery obstruction groups \cite[Chap.~10]{wall-book} and $\CTNI(\XX )$ is the set of tangential 
normal invariants of $\XX$.

\begin{remark}
The definition of $\CTNI(\XX )$ is similar to the definition of $\CTSS(\XX )$ except that for representatives
$(\MM , f, b)$ we require only that $f \colon \MM  \to \XX $ is a degree one map.  The equivalence relation,
often called {\em normal cobordism}, is defined using
normal cobordisms over $(\XX , \nu_\XX )$.
% and not just normal $s$-cobordism over $(\XX , \nu_\XX )$.
In other words, $\CTNI(\XX)$ is the bordism set 
$\Omega_m(\XX, \nu_\XX)_1 \subset \Omega_m(\XX, \nu_\XX)$ 
of normal cobordism classes of normal $(\XX, \nu_\XX)$-manifolds $(M, f, b)$ as defined in 
\cite[Chapter II]{stong-book},
where in addition $f \colon \MM \to \XX$ has degree one.
\end{remark}

Let $T(\nu_\XX )$ denote the Thom-space of $\nu_\XX$
and  $\rho_M\colon S^{m+k} \to T(\nu_M)$ denote the (canonical) collapse map arising from a stable embedding of $M^m \subset S^{m+k}$.  The Pontrjagin-Thom isomorphism,
\[ \pt \colon \Omega_m(\XX, \nu_\XX) \cong \pi_{m+k}(T(\nu_\XX)),  \quad [\MM , f, b] \mapsto [T(b)\circ \rho_M],  \]
identifies the bordism group of normal $(\XX, \nu_\XX)$-manifolds (of any degree) with the stable homotopy group
 of $T(\nu_\XX)$.  Here $T(b) \colon T(\nu_\MM) \to T(\nu_\XX)$ is the map of Thom spaces induced by the
 bundle map $b \colon \nu_\MM \to \nu_\XX$. %and $f$ is not assumed to be of degree one,
 This isomorphism specialises to the bijection
\[ \pt \colon \CTNI(\XX ) = \Omega_m(\XX , \nu_\XX )_1 \cong \pi_{m+k}(T(\nu_\XX ))_1, \]
%\quad [\MM , f, b] \mapsto [T(b)\circ \rho_M], \]
%
 where the subscript $1$ indicates the the pre-image of $1 \in \bbZ$ under the Thom maps 
\[ \Omega_m(\XX , \nu_\XX ) \to H_m(\XX ; \bbZ) \quad \text{and} \quad \pi_{m+k}(T(\nu_\XX )) \to H_m(\XX ; \bbZ). \]

Spanier-Whitehead duality, henceforth $\sw$-duality, defines a contravariant functor on the stable homotopy category of stable finite CW complexes:
see, for example \cite[I.4]{browder72}.  
Recall that the $\sw$-dual of $T(\nu_\XX )$ is $\XX _+$, the disjoint union of $X$ and a point.
Given a map  
$$\rho \colon  S^{m+k} \to T(\nu_\XX ),   $$
the $\sw$-\emph{dual} of $\rho$ is a stable map
$ \SW(\rho) \colon \XX _+ \to S^0 $
and the \emph{adjoint} of $\SW(\rho)$ is a map
$ \DD(\rho) \colon \XX  \to \qsn $,
where $\qsn = \Omega^\infty S^\infty$ has its usual meaning.  
In particular, ``degree" defines a homomorphism 
$\pi_0(\qsn) \cong \bbZ$ and we let $\qsna_a$ be the $a$-th component of $\qsn$. 
The space $\qsn$ is an $H$-space under the loop product $\ast \colon \qsn \times \qsn \to \qsn$ which
satisfies
\[   \ast \colon \qsna_a \times \qsna_b \to \qsna_{a+b}, \]
and for any space $\XX$ there is a free and transitive action
\[ [\XX , \qsna_1] \times [\XX, \qsna_0] \xra{\ast} [\XX , \qsna_1] \quad ([\varphi], [\alpha]) \mapsto [\varphi] \ast [\alpha] .\]

\begin{lemma} \label{lem:normal_invariant_1}
%%%%%%%%
There is an isomorphism of abelian groups,
\[  \DD \colon \pi_{m+k}(T(\nu_\XX )) \cong [\XX , \qsn], \quad [\rho] \mapsto [\DD(\rho)],  \]
such that
\begin{enumerate}
\item $\DD(\pi_{m+k}(T(\nu_\XX ))_a) = [\XX , \qsna_a]$, \label{it:Ddeg}
\item $\DD(\pt[\XX , \Id, \Id]) = [1]$, the constant map at the identity in $\qsna_1$. \label{it:Did}
\end{enumerate}
\end{lemma}

\begin{proof}
That $\DD$ is an additive isomorphism follows from the properties of $\sw$-duality and the 
adjoint correspondence.  In particular, the loop product corresponds to the addition of stable 
maps under $\sw$-duality and passing to adjoints.  

(i) Let $c_{S^{m+k}} \colon  T(\nu_{\XX}) \to S^{m+k}$ be the degree one collapse to the top
cell of the Thom space.
Given a map $\rho \colon S^{m+k} \to T(\nu_{\XX})$ the degree of $c_{S^{m+k}} \circ \rho$ is the degree
of the normal map corresponding to $\rho$.  But the $\sw$-dual of $c_{S^{m+k}}$ is
the inclusion of the base-point $+ \to \XX_+$ and hence the degree of $c_{S^{m+k}} \circ \rho$
is given by the component of $\DD(\rho)$ in $\qsn$.

(ii) This is an exercise is $\sw$-duality.
By \cite[Theorem I.4.13]{browder72}, two spaces $A$ and $A'$ are $\sw$-dual if and only if
there is a map $\lambda \colon S^d \to A \wedge A'$ such that slant product with $\lambda_*([S^d])$
induces an isomorphism $H^q(A) \cong H_{d-q}(B)$ for all $q$.
An elegant duality map for the $\sw$-dual pair $(T(\nu_{\XX}), \XX_+)$ is the 
``Atiyah duality map'' as described in \cite[\S 3]{klein00}.
Let $\rho_\XX \colon S^{m+k} \to T(\nu_{\XX})$ be the Thom collapse map and let 
$T(\Delta_{\XX}) \colon T(\nu_{\XX}) \to T(\nu_{\XX}) \wedge \XX_+$ 
by the map of Thom spaces induced by the bundle map
\[ \Delta_{\XX} \colon  \nu_\XX  \to {\rm pr}_1^*(\nu_{\XX}) \]
where ${\rm pr}_1 \colon \XX \times \XX \to \XX$ is the projection to the first factor.
Then $\lambda_{\XX} : = \rho_{\XX} \circ T(\Delta_{\XX})$ is an $m$-duality map for $(T(\nu_{\XX}), M_+)$.

Now let $c_{S^0} \colon \XX_+ \to S^0$ be the collapse map collapsing $\XX$ to a point and preserving base-points.
There this is a commutative diagram,
\[  \xymatrix{ S^{m+k} \ar[d]^{\lambda_{\XX}}  \ar[rr]^(0.4){\Id} & &  S^{m+k} \wedge S^0 \ar[d]^{\rho_{\XX} \wedge \Id} \\
T(\nu_{\XX}) \wedge \XX_+ \ar[rr]^{\Id \wedge c_{S^0}} & & T(\nu_{\XX}) \wedge S^0,  }   \]
and so by \cite[Theorem I.4.14]{browder72}, $c_{S^0}$ is the $\sw$-dual of $\rho_{\XX} = \pt([\XX, \Id, \Id])$.  
The adjoint of $c_{S^0}$ is the constant map at $[1]$ and this completes the proof. 
\end{proof}

By definition $\qsna_1 = SG$, the space of stable orientation-preserving self-homotopy equivalences of 
the sphere.   We define the tangential normal invariant to be the map
\eqncount
\begin{equation} \label{eq:normal_invariant}
 \Nt \colon \CTNI(\XX ) \longrightarrow [\XX , SG], \quad [\MM , f, b] \longmapsto \DD\bigl( \pt([\MM , f, b]) \bigr) .
\end{equation}
By Lemma \ref{lem:normal_invariant_1} we see that $\Nt$ is a set bijection such that
$\Nt([\XX , \Id, \Id]) = [1]$.  
The following lemma is a direct consequence of the definition of $\Nt$ and Lemma \ref{lem:normal_invariant_1}.

\begin{lemma} \label{lem:normal_invariant_2}
Let $[P, h, b] \in \Omega_m(\XX , \nu_\XX )_0$ with $\pt([P, h, b]) = \rho_b \in \pi_{m+k}(T(\nu_\XX))_0$.
Then 
$$ \Nt \bigl( [\XX , \Id, \Id] + [P, h, b] \bigr) = [1] \ast \DD(\rho_b). $$
\end{lemma}

We next prove a lemma about the behaviour of the tangential normal invariant along sub-manifolds.
Let $t \colon \YY \subset \XX$ be the inclusion of a closed submanifold of codimension $a > 0$ and let 
$(f, b) \colon \MM \to \XX$ be a tangential degree one normal map.  Taking the transverse inverse image along 
$\YY$ induces a well-defined map of normal invariant sets
\[  \pitchfork_{\YY} \colon \CTNI(\XX ) \to \CTNI(\YY), \quad [\MM, f, b] \mapsto 
[f^{-1}(\YY), f|_{f^{-1}(\YY)}, b_{\YY, f} \oplus b|_{f^{-1}(_\YY)} ] \]
where $b_{\YY, f} \colon \nu_{f^{-1}(\YY) \subset \MM} \to \nu_h$ is the canonical bundle map given by the implicit function theorem.

\begin{lemma} \label{lem:transverse}
%%%%%%%%
The map $\pitchfork_{\YY} \colon \CTNI(\XX ) \to \CTNI(\YY)$ fits into the following commutative diagram:
\[ \xymatrix{  \CTNI(\XX ) \ar[r]^(0.5){\Nt} \ar[d]^{\pitchfork_{\YY}} & [\XX, SG] \ar[d]^{j^*} \\
 \CTNI(\YY) \ar[r]^(0.5){\Nt}  & [\YY, SG] .}  \]
\end{lemma}

\begin{proof}

Consider the ``wrong way'' map of Thom spaces 
$ \wh j \colon T(\nu_\XX) \to T(\nu_\YY) $
induced by the embedding $j \colon \YY \subset \XX$.  It follows from the definitions of the Pontrjagin-Thom
isomorphism $\pt$ and the duality isomorphism $\DD$ that there is a commutative diagram,
\[  \xymatrix{ \CTNI(\XX) \ar[d]^{\pitchfork_{\YY}} \ar[r]^(0.45){\pt}&
 \pi_{m+k}(T(\nu_{\XX}))_1 \ar[d]^{\wh j_*} \ar[r]^(0.575){\DD} & [\XX, SG] \ar[d]^{j^*}  \\
\CTNI(\YY) \ar[r]^(0.4){\mu_Y} & \pi_{m-a+k}(T(\nu_{\YY}))_1 \ar[r]^(0.6){\DD} & [Y, SG].   } \]
The lemma now follows since by definition $\Nt = \DD \circ \pt$ and similarly for $\DD \circ \mu_Y$.
% and $\Nt = \DD \circ \mu_Y$.
\end{proof}

We conclude this section by recording the relationship between tangential surgery and classical surgery.
We assume that the reader is familiar with classical surgery as described in \cite{wall-book} and in particular
with the identification of the usual normal invariant set
\[   \Nv \colon \CNI(\XX) \equiv [\XX, G/\cat] . \]
There are natural maps from the tangential surgery exact sequence of \eqref{eq:ses} to the usual
surgery exact sequence
\eqncount
\begin{equation} \label{eq:FSES}
\xymatrix{
L_{m+1}(\bZ\pi) \ar[r]^{\theta} \ar[d]^= & \CTSS(\XX ) \ar[r]^(0.5)\Nt \ar[d]^{} & [\XX, SG] \ar[r]^\sigma \ar[d]^{i_*} & L_m(\bZ\pi) \ar[d]^=\\
L_{m+1}(\bZ\pi) \ar[r]^{\theta} & \CSS(\XX ) \ar[r]^(0.45)\Nv & [\XX, G/\cat] \ar[r]^(0.55)\sigma & L_m(\bZ\pi).
}
\end{equation}
Here we have replaced $\CTNI(\XX)$ with $[\XX, SG]$ using $\Nt$, and $i_*$ is the map induced by the canonical map
$i \colon SG \to G/\cat$ (see \cite[(2.4)]{madsen-taylor-williams1}).

\section{The normal invariants of pinch maps} \label{sec:normal_invariants_of_pinch_maps}
%%%%%%%%%%%%%%%%%%%%%%%%%%%%%%%%%%%%%%%%%%%%%%
%
In this section we consider the normal invariants of tangential self homotopy equivalences $(\XX, p, b)$
covering certain pinch maps $p \colon \XX \simeq \XX$.
Let $t \colon \YY  \subset \XX $ be the inclusion of closed codimension $l > 0$ submanifold $\YY $
in a closed $m$-manifold $\XX $, 
 in either the smooth or PL categories.  Let $\nu_t$ be the
normal bundle of $t(\YY ) \subset \XX $ so the stable normal bundle of $\YY $ is given by
\eqncount
\begin{equation} \label{eq:normal_bundles}
\nu_\YY  = \nu_t \oplus t^*(\nu_\XX) . 
\end{equation}
A key map in the following will be the collapse map 
\[  t^!_+ \colon \XX \to T(\nu_t) \]
which collapses $\XX$ to the Thom space of $\nu_t$, $T(\nu_t)$, and maps $+$ to
the base-point of $T(\nu_t)$.  We suppose that are given a map
$ y \colon S^m \to \YY  $
such that the composite $x = t \circ y$,
\[ x \colon S^m \xra{y} Y \xra{t} \XX, \]
 pulls back $\nu_\XX $ trivially.  Since $\nu_{S^m}$ is trivial, this is equivalent to assuming
the existence of a bundle map $b_y \colon \nu_{S^m} \to t^*(\nu_\XX) $.   If $b_t \colon t^*(\nu_\XX)  \to \nu_\XX $ 
is the canonical bundle map, we set $b_x := b_t \circ b_y$ and consider the following diagram
of bundle maps:
\[  \xymatrix{  \nu_{S^m} \ar[d] \ar[r]^(0.4){b_y} & t^*(\nu_\XX)  \ar[r]^(0.55){b_t} \ar[d] & \nu_\XX  \ar[d] \\
S^m \ar[r]^y & \YY \ar[r]^t  & \XX  }  \]
The homotopy class $\rho_x  := \pt([S^m, x, b_x])$ is then given as the composite
\eqncount
\begin{equation}   \label{eq:Thom}
\rho_x = (T(b_t) \circ \rho_y )\colon S^{m+k} \xra{~\rho_y~} T(t^*(\nu_\XX) ) \xra{T(b_t)} T(\nu_\XX )
\end{equation}
where $\rho_y$ is the homotopy class $T(b_y)_*(\rho_{S^m}) \in  \pi_{m+k}(T(t^*(\nu_\XX) ))$
and $T(b_t)$ and $T(b_y)$ denote the induced maps of Thom spaces. 
Since $\rho_x$ has degree zero, we have the map $\DD(\rho_x) \colon \XX \to \qsna_0$.  To analyse $\DD(\rho_x)$
we consider the $\sw$-duals of the maps in \eqref{eq:Thom}.

\begin{lemma} \label{lem:Tnut}  \hfill
\begin{enumerate}
\item \label{it:SWT(b_t)}
The $\sw$-dual of $T(b_t) \colon T(t^*(\nu_\XX)) \to T(\nu_\XX)$ is given by
the collapse map of $t$;
$\SW (T(b_t))  =  t^!_+ \colon \XX_+ \to T(\nu_t). $
\item \label{it:DTnut}
$  \DD \colon \pi_{m+k}(T(t^*(\nu_\XX) )) \cong [T(\nu_t), \qsna_0].  $
\item \label{it:Dcx_and_Dcy}
$ \DD(\rho_x) = \DD(\rho_y) \circ t^{!} \in [\XX, \qsna _0] . $
\end{enumerate}
\end{lemma}

\begin{proof}
(i) From the bundle identity $\nu_\YY = \nu_t \oplus t^*(\nu_\XX)$ of  \eqref{eq:normal_bundles},
we have by  \cite[Theorem 3.3]{atiyah2} that 
\[  \SW \bigl( T(t^*(\nu_\XX)) \bigr) \simeq T(\nu_\YY \ominus t^*(\nu_\XX)) \simeq T(\nu_t) .\]
This duality can be realised as follows.  Start with the bundle map $\Delta \colon \nu_\YY \to t^*(\nu_\XX) \times \nu_t$
which covers the diagonal map $Y \to Y \times Y$ and take the composition
\[ \lambda_{\YY, \nu_t} : =  T(\Delta) \circ \rho_\YY \colon S^{m+k} \to  T(\nu_\YY) \to T(t^*(\nu_\XX)) \wedge T(\nu_t) .\]
To verify that $\SW(T(b_t)) = t^!$ we shall show that the following diagram commutes up to homotopy.
\[   \xymatrix{    S^{m+k} \ar[d]^{\lambda_\XX} \ar[r]^{\rho_\XX} & T(\nu_\XX) \ar[r]^{\wh t} & T(\nu_\YY) \ar[r]^(0.325){\lambda_{Y, \nu_t}} & T(t^*(\nu_\XX)) \wedge T(\nu_t) \ar[d]^{T(b_t) \wedge \Id} \\
T(\nu_\XX) \wedge X_+ \ar[rrr]^{\Id \wedge t^!} & & & T(\nu_\XX) \wedge T(\nu_t)   }  \]
Going in either direction around the diagram gives an element of 
\[ \pi_{m+k}(T(\nu_\XX) \wedge T(\nu_t)) \cong \pi_{m+k}(T(\nu_\XX \times \nu_t)) \cong 
\Omega_{m-l}(X \times Y; \nu_\XX \times \nu_t) \]
where the last isomorphism is the Pontrjyagin-Thom isomorphism.  We claim that, in both directions, 
the corresponding normal $(X \times Y, \nu_\XX \times \nu_t)$-manifold is $(Y, t \times \Id_\YY, b_Y)$, where
$b_Y \colon \nu_\YY \to \nu_\XX \times \nu_t$ is the canonical bundle map defined by the bundle
isomorphism $\nu_\YY \cong \nu_t \oplus t^*(\nu_\XX)$.  

For the composition 
$(\Id \wedge t^!) \circ \lambda_\XX$, the homotopy class $\lambda_\XX$ corresponds to the element of 
$\Omega_m(\XX \times \XX; {\rm pr}_1^*(\nu_\XX))$ given by $[\XX, \Delta_\XX, \Id]$; here ${\rm pr_1}$
is the projection to the first factor of $\XX \times \XX$.  Moreover, the map $\Id \wedge t^!$ corresponds to taking 
the transverse inverse image of along $\XX \times \YY \subset \XX \times \XX$ and so maps
$(\XX, \Delta_\XX, \Id)$ to $(Y, t \times \Id_\YY, b_\YY)$.   

For the other composition, we start by noting that $\wt t \circ \rho_\XX = \rho_\YY$ and $\rho_\YY$
is the stable homotopy element defined by the bordism class of $(\YY, \Id, \Id)$ in $\Omega_{m-l}(\YY, \nu_\YY)$.  
Since $\lambda_{\YY, \nu_t}$ is the map of Thom spaces induced by the bundle map $\Delta$ and $b_t \colon t^*(\nu_\XX) \to \nu_\XX$ is the canonical bundle map,
we see that $[\YY, \Id, \Id]$ is mapped to $[Y, t \times \Id_\YY, b_Y]$.

(ii) This the analogue of the bijection in Lemma \ref{lem:normal_invariant_1}.

(iii) This follows immediately from the definition of the duality map $\DD$ and part (i).
\end{proof}

Recall from Section \ref{sec:pinch} that the map $x = t \circ y \colon S^m \to \XX$ 
can be used to define a self-homotopy equivalence
$p(t \circ y) \colon \XX \simeq \XX$, the pinch map on $x$.

\begin{lemma} \label{lem:pinch}
%%%%%%%
There is a bundle map $b \colon \nu_\XX \to \nu_\XX$ covering $p(t \circ y) \colon \XX  \simeq \XX $ such that 
\[ \Nt ([\XX , p(t \circ y), b]) = [1] \ast  \DD(\rho_x) = [1] \ast \bigl(\DD(\rho_y) \circ t^{!} \bigr) \in [\XX, SG]. \]
\end{lemma}
\begin{proof}
%%%%%%%
Consider the degree one normal map $(\XX, \Id, \Id)  \sqcup (S^m, x, b_x)$. 
The connected sum of normal $(\XX , \nu_\XX )$-manifolds is a well-defined operation which preserves the 
$(\XX , \nu_\XX )$-bordism class.  This is because we
may assume that there are embedded discuss $D^m \subset S^m$ and $D^m \subset \XX$ 
such that $x \colon S^m \to \XX$ maps $D^m$ identically to $D^m$.
%and $x(S^m - D^m) \subset \XX - D^m$ 
Performing zero surgery on $D^m \sqcup D^m \subset \XX \sqcup S^m$ over $(\XX , \nu_\XX )$, i.e.~taking connected sum of normal $(\XX , \nu_\XX )$-manifolds, we see from the definition of a pinch map that
\[  (\XX , \Id, \Id) \sharp (S^m, x, b_x) = (\XX , p(x), b), \]
where $b \colon \nu_\XX  \to \nu_\XX $ is some bundle map covering $p(x)$.  It follows that
\[ [\XX , p(x), b] = [\XX , \Id, \Id] + [S^m, x, b_x] \in \Omega_m(\XX , \nu_\XX )_1 .\]
Applying Lemma \ref{lem:normal_invariant_2} proves the first equality of the lemma.
The second equality follows from Lemma \ref{lem:Tnut} \eqref{it:Dcx_and_Dcy}.
\end{proof}

%%%%%%%%%%%%%%%%%%%%%%%%%%%%%%%%%%%%%%%%%%%%%%%%%

\section{The proof of Theorem B}
\label{sec:induct}
We first  outline the ingredients involved in the proof of Theorem B, for a given dimension $4k{+}2 = \dim \theta_{j+1} = 2^{j+2} -2$.  Let $n = 2k = \dim \theta_j$.

\begin{enumerate}
\item Let $\xi$ be an admissible PL-bundle of dimension $n+1$ over
$\PP^{n+1}$, as in Definition \ref{def:admissible}. We have $\pi^*(\xi) \cong \tau_{S^{n+1}}$. Let $\WW = D(\xi)$ and $\VV = \bd \WW = S(\xi)$.
\item Suppose that there exists an element  $x_j \in \theta_j$, with $2x_j = 0$ and Kervaire invariant one. By Theorem \ref{thm:B-J-M}, this occurs if and only if 
$[\iota_{n+1}, \iota_{n+1}] = 2\alpha$, for some class
 $\alpha \in  \pi_{2n+1}(S^{n+1})$ such that $x_j = \Sigma(\alpha)$, where $\Sigma\colon \pi_{2n+1}(S^{n+1}) \to \pi^S_n$ is the suspension homomorphism. 
 \item Let $\pa\colon \VV \to \VV$ denote the pinch map defined in 
 Definition \ref{def:p_alpha}.
\end{enumerate}

The main result to be proven in this section is the following:

\begin{theorem} \label{thm:M} Suppose that $[\iota_{n+1}, \iota_{n+1}] = 2\alpha$, for some  $\alpha \in  \pi_{2n+1}(S^{n+1})$, with $\Sigma(\alpha) = x_j \in \theta_j$.  Then the pinch map $\pa $ is homotopic to a PL-homeomorphism $g\colon  \VV  \cong \VV $.
\end{theorem}

We obtain the involutions of Theorem B by constructing their quotients.
These quotients are PL-twisted doubles,
\eqncount
\begin{equation}\label{eq:M}
M: = M(\xi, \alpha, g) = D(\xi) \cup_g D(\xi),
\end{equation}
where $g$ is a PL-homeomorphism provided by Theorem \ref{thm:M}
and $\xi$ and $\alpha$ are as above.
%The notation $M(\xi, \alpha, g)$ emphasizes the data used in the construction
%of the quotients.}
We must of course identify the universal cover of $M$ and for this we have.

\begin{proposition} \label{prop:MK=tildeM}
The closed PL-manifold $\MM =\MM(\xi, \alpha, g)$ has universal covering PL-homeomorphic to $\MK^{4k{+}2}$, with an involution of type $\xi$.
\end{proposition}
\noindent
To prove Proposition \ref{prop:MK=tildeM} we need the following application of simply-connected surgery. 

\begin{lemma}\label{lem:Kervaire_manifold}
%%%%%%%%
Any homotopy equivalence $f \colon \NN \to \MK^{4k{+}2}$ from a closed PL-manifold $\NN$ to a Kervaire manifold is homotopic
to a PL-homeomorphism.
\end{lemma}

\begin{proof}
Since $L_{4k+3}(\bZ) = 0$, the PL-surgery exact sequence for $\MK^{4k{+}2}$ runs as follows:
\[  0 \to \cS_{PL}(\MK^{4k{+}2}) \xra{~\Nv~} [\MK^{4k{+}2}, G/PL] \xra{~\sigma~} L_{4k{+}2}(\bZ) \to 0 \]
From Section \ref{sec:pinch} there is a homotopy equivalence 
$\MK^{4k{+}2} \simeq (S^{2k{+}1}_0 \vee S^{2k{+}1}_1) \cup_\varphi D^{4k{+}2}$
where $\varphi \colon S^{4k+1} \to S^{2k{+}1}_0 \vee S^{2k{+}1}_1$ is a stably 
trivial map.  As $\pi_{2k{+}1}(G/PL) = 0$, it follows that the collapse map
$c_{\MK} \colon \MK^{4k{+}2} \to S^{4k{+}2}$ induces an isomorphism 
$c_{\MK}^* \colon \pi_{4k{+}2}(G/PL) \cong [\MK^{4k{+}2}, G/PL]$.  But $\sigma \circ c_{\MK}^*$ 
is an isomorphism and $\Nv$ is injective.
Hence $\cS_{PL}(\MK^{4k{+}2})$ has one element which proves the lemma.
\end{proof}

\begin{proof}[Proof of Proposition \ref{prop:MK=tildeM}]
By Theorem \ref{thm:homotopy_type},
our assumptions on $\xi$ and $\alpha$ 
imply that $\wt\MM$ is homotopy equivalent to $\MK^{4k{+}2}$. 
By Lemma \ref{lem:Kervaire_manifold}, $\wt\MM$ is PL-homeomorphic to $\MK^{4k+2}$.
%It is therefore enough to show that any manifold homotopy equivalence to
%the PL-Kervaire manifold $\MK$ is PL-homeomorphic to it.
 \end{proof}

Assuming Theorem \ref{thm:M}, we now have the following result,
which implies Theorem B.

\begin{theorem}\label{thm:thmb_detail}
%%%%%%%%
Suppose that the set $\theta_j$ contains an element of order two, for some $j\geq 0$. 
If $\xi$ is an admissible PL-bundle of dimension $2k{+}1$ over $\PP^{2k{+}1}$, 
with $k = 2^j -1$, then $\MK^{4k{+}2}$ admits a free orientation-preserving (PL) involution of type $\xi$.
\end{theorem}

\begin{proof}
Let $M = M(\xi, \alpha, g)$ by the PL-manifold which we recall is the twisted double 
$M = D(\xi) \cup_g D(\xi)$.
By Proposition \ref{prop:MK=tildeM}, there is a PL-homeomorphism 
$f \colon \wt\MM \cong \MK^{4k+2}$.  Hence if $\tau \colon \wt \MM \cong \wt \MM$ 
denotes the non-trivial deck transformation of $\wt \MM \to \MM$,
then the PL-homeomorphism $f^{-1} \circ \tau \circ f \colon \MK^{4k+2} \cong \MK^{4k+2}$ 
is free orientation-preserving PL-involution of type $\xi$
on $\MK^{4k+2}$.
\end{proof}

\begin{remark} Theorem \ref{thm:thmb_detail}
shows that there exist many inequivalent PL-involutions on the Kervaire manifolds, 
just by varying the choice of characteristic bundle $\xi$.
\end{remark}
%
%\begin{lemma}  \label{lem:p_alpha_PL}
%%%%%%%%
%The pinch map $\pa  \colon  \VV  \simeq \VV $ in Definition \ref{def:p_alpha} is  homotopic to a PL-homeomorphism.  
%%
%\end{lemma}
%
\begin{proof}[The proof of Theorem \ref{thm:M}]
%%%%%%
It is enough to show that the pinch map $\pa  \colon \VV  \simeq \VV $ is equivalent to 
$\Id \colon \VV  \simeq \VV $ in $\cS_{PL}(\VV )$.
Now $\VV $ is an orientable manifold with $\pi_1(\VV ) = \bbZ/2$, 
and by \cite[\S 13.A]{wall-book} the map $L_{2n{+}2}(\bZ) \to L_{2n{+}2}(\bZ[\bbZ/2],+)$ is an isomorphism.  
Since the $L$-groups of the trivial group act trivially on any PL-structure set, $\cS_{PL}(\VV )$ injects into $\cN_{PL}(\VV )$.  So we must prove that the usual PL-normal invariant
$\varphi : = \Nv(\pa)\colon \VV  \to G/PL$ vanishes.  

By Lemma \ref{lem:pa_is_tangential}, 
there is a bundle map $b \colon \nu_\VV \to \nu_\VV$ covering $\pa$ 
and so by diagram \eqref{eq:FSES}, $\varphi = i_{} \circ \Nt(b)$.
Now from Lemma \ref{lem:pinch}, the normal invariant of $\pa $ factors as follows
\[  \varphi = \psi \circ s^! \colon \VV  \xra{~s^!~}  T(\nu_s) \xra{D(\rho_{b})} \qsna_0 \xra{[1] \ast} SG \xra{~i~} G/PL.  \]
where $\psi : = i \circ ([1]\ast) \circ D(\rho_b)$ and $i \colon SG \to G/PL$ is the canonical map.
As the bundle $\nu_s$ has rank $n$, the Thom space $T(\nu_s)$ is $(n-1)$-connected and so
$\varphi$ vanishes on the $(n-1)$-skeleton of $\VV $.  
%Since $n>4$, this allows us to ignore subtleties with the lower Postnikov stages of $G/PL$.  
It follows that the map $\psi \colon T(\nu_s) \to G/PL$ lifts to a map 
$T(\nu_s) \to G/PL\an{n}$.
%$\psi\an{n} \colon T(\nu_s) \to G/PL\an{n}$ and we shall show that $\psi\an{n}$ is null-homotopic
 
Because there is an odd-primary equivalence $T(\nu_s)_{(odd)} \simeq S^{2n+1}$, it will be sufficient
to work $2$-locally.  There are isomorphisms
\[ [T(\nu_s), G/PL\an{n}] \cong [T(\nu_s), G/PL\an{n}]_{(2)} \cong [T(\nu_s), G/PL\an{n}_{(2)}]. \]

Turning to the $2$-local situation, by \cite[Lemma 4.7]{madsen-milgram1} there are cohomology classes
$\kappa_{4k{+}2} \in H^{4k{+}2}(G/PL\an{n}; \bbZ/2)$ and 
$\bar \kappa_{4k} \in H^{4k}(G/PL\an{n}; \bbZ_{(2)})$ such that the map 
\[  \prod_{4k{+}2 \geq 6} ( \kappa_{4k{+}2} \times \bar \kappa_{4k+4} ) 
\colon G/PL\an{6} \simeq 
\prod_{4k +2 \geq 6}  K(\bbZ/2, 4k{+}2) \times K(\bbZ_{(2)},4k+4) \]
is a $2$-local homotopy equivalence. 
It follows that $[T(\nu_s), G/PL\an{n}]$ can be expressed as a direct sum of cohomology groups:
\[ [T(\nu_s), G/PL\an{n}_{}]  \cong \bigoplus_{4k{+}2 \geq n} H^{4k{+}2}(T(\nu_s); \bbZ/2) \oplus  H^{4k+4}(T(\nu_s); \bbZ_{(2)}).   \]
Since mod~$2$ reduction $\rho_2 \colon H^{4k+4}(T(\nu_s); \bbZ_{(2)}) \to H^{4k+4}(T(\nu_s); \bbZ/2)$ is an isomorphism it
will suffice to consider the cohomology classes $\kappa_{4k+4} : = \rho_2 \circ \bar \kappa_{4k+4}$ and $\kappa_{4k{+}2}$.

We need to show that $\psi^*(\kappa_{2a})= 0$ for each $a \geq n/2$.  Since $\psi$ factors through the map 
$i \colon SG \to G/PL$,
we can use a deep result of Brumfiel, Madsen and Milgram about the induced map of mod~$2$ cohomology 
$i^* \colon H^*(G/PL; \bbZ/2) \to H^*(SG; \bbZ/2)$.

\begin{theorem}[{\cite[Corollary 3.4]{brumfiel-madsen-milgram1}}] \label{thm:B-M-M}
%%%%%%%%
$i^*(\kappa_{2a}) = 0$ if $a \neq 2^k$ or $2^k -1$ and $i^*(\kappa_{2^k+1}) = i^*(\kappa_2^{2k})$.
\end{theorem}
\noindent
Since $T(\nu_s)$ is an $n$-connected $(2n+1)$-dimensional CW-complex, 
%and $\psi$ factors through $i$, $\psi = (([1]\ast) \circ D(c_\alpha)) \circ i$,
the first part of Theorem \ref{thm:B-M-M} implies that the only possible non-zero classes 
$\psi^*(\kappa_{2a}) \in H^*(T(\nu_s); \bbZ/2)$ are 
$\psi^*(\kappa_{n})$ and $\psi^*(\kappa_{n{+}2})$.
But by the second part of Theorem \ref{thm:B-M-M}, $\psi^*(\kappa_{n{+}2}) = (\psi^*(\kappa_2))^{j+1} = 0$.

To show that $\psi^*(\kappa_n) = 0$, we use the surgery-theoretic definition of the $\kappa$-classes.
We give the relevant formula only in the special case we need.
Let $\XX $ be a closed connected $(4k{+}2)$-dimensional PL-manifold with trivial total Wu class, $v(\XX ) = 1 \in H^*(\XX; \bbZ/2)$,
and let $(f, b) \colon \MM \to \XX$ be a degree one normal map 
with normal invariant the map  $\theta \colon \XX  \to G/PL$.  Then by \cite[(2.6)]{brumfiel-madsen-milgram1}, 
\eqncount
\begin{equation} \label{eq:surgery_formula}
\sigma_2([\MM , f, b]) = \an{\theta^*(\kappa_{4k{+}2}), [\XX]},
\end{equation}
where $\sigma_2([\MM, f, b]) \in \bbZ/2$ is the mod~$2$ surgery obstruction of $[\MM, f, b]$.
We shall apply this formula to compute $\psi^*(\kappa_n) \in H^n(T(\nu_s); \bbZ/2) \cong \bbZ/2$.  The generator
of $H^n(T(\nu_s); \bbZ/2)$ is the Thom class of $T(\nu_s)$ which is Poincar\'{e} dual to the fibre $n$-disc of
the bundle $\nu_s$.  It follows that the Poincar\'{e} dual of the pull-back 
$(s^{!})^*\psi^*(\kappa_n) = \varphi^*(\kappa_n)$ is represented by the inclusion of a fibre $f \colon S^n \hookrightarrow \VV $.
By Lemma \ref{lem:transverse} and the diagram \eqref{eq:FSES},  taking the transverse inverse image along $S^n$ 
defines the homomoprhism $\pitchfork_{S^n}$ in 
%gives a well-defined map of normal invariant sets
%$\pitchfork_{S^n} \colon \cN_{PL}(\VV ) \to \cN_{PL}(S^n)$ which fits into 
the following commutative diagram:
\[ \xymatrix{  \cN_{PL}(\VV ) \ar[r]^(0.45){\Nv} \ar[d]^{\pitchfork_{S^n}} &  [\VV, G/PL] \ar[d]^{f^*}\\
 \cN_{PL}(S^n) \ar[r]^(0.425){\Nv} & [S^n, G/PL] .}  \]
We wish to understand $\an{\varphi^*(\kappa_n), f_*[S^n]} = \an{f^*\varphi^*(\kappa_n), [S^n]}$.  
Since $\Nv^{-1}(\varphi) = [\VV , \pa , b]$ and $v(S^n) = 1$, it suffices to compute the surgery obstruction 
\[  \sigma_2(\pitchfork_{S^n}\!\!([\VV , \pa , b])) \in \bbZ/2. \]
Recall that $\pa  = p(s \circ \pi \circ \alpha)$ is the pinch map on the composition
\[ S^{2n+1} \xra{~\alpha~} S^{n+1} \xra{~\pi~} \PP^{n+1} \xra{~s~} \VV . \]
We may assume that $f(S^n)$ is disjoint from the cite of the pinching.  Since $s(\PP^{n+1})$ and $f(S^n)$
meet transversely in a single point $v \in \VV$, it follows that $\pa $ is transverse to $f(S^n) \subset \VV $ with inverse image
\[ \pa ^{-1}(f(S^n)) = f(S^n) \sqcup (s \circ \pi \circ \alpha)^{-1}(v). \]
As $\pi$ is the standard double covering, $\pi^{-1}(v)$ is a pair of antipodal points $v_0 \sqcup v_1 \in S^{n+1}$.  We may assume
 that $\alpha^{-1}(v_0) = \alpha^{-1}(v_1)$
and that $(s \circ \pi \circ \alpha)^{-1}(v) = \alpha^{-1}(v_0) \sqcup \alpha^{-1}(v_1)$ is a disjoint union
of diffeomorphic manifolds $\alpha^{-1}(x_0) \cong \alpha^{-1}(v_1)$ with diffeomorphic framings 
$F_0$ and $F_1$ covering the constant maps $c_i \colon \alpha^{-1}(v_i) \to v \in \VV$, $i = 0,1$.  
It follows that
\[  \sigma_2(\pitchfork_{S^n}\!\!([\VV , \pa , b])) = 2 \sigma_2([\alpha^{-1}(v_0), c_0, F_0]) = 0 .\]
Applying the surgery formula \eqref{eq:surgery_formula} we deduce
that $\an{f^*\varphi^*(\kappa_n), [S^n]} = 0$.  It follows that $\psi^*(\kappa_n) = 0$ and so $[\psi] = 0 \in [T(\nu_s), G/PL]$.
Since $\varphi  = s^! \circ \psi$, we conclude that $\Nv([\VV, \pa , b]) = [\varphi] = 0 \in [\VV , G/PL]$ and we
are done.
\end{proof}
%%%%%%%%%%%%%%%%%%%%%%%%%%%%%%%%%%%%%%%%%%%%%%
 %\newpage

\section{The proof of Theorem C}\label{sec:odd order}
%%%%%%%%%%%%%%%%%%%%%%%%%%%%%%%%%%%%%%%%%%%%%%
We will now compare free finite group actions on $\MK^{4k{+}2}$ and $S^{2k{+}1} \times S^{2k{+}1}$. Since the Kervaire manifolds usually do not admit free involutions, we will consider odd order group actions.
Recall from Section \ref{sec:pinch} that the homotopy type of both manifolds has the form
$$(S^{2k{+}1} \vee S^{2k{+}1}) \cup D^{4k{+}2},$$
and the attaching maps of the top cell differ only by the addition of the  
Whitehead square $w = [\iota_{2k{+}1}, \iota_{2k{+}1}] \in \pi_{4k+1}(S^{2k{+}1})$. 
The Whitehead square has order two and Hopf invariant zero, so we may construct a degree four map
$$f\colon \MK^{4k{+}2}\to S^{2k{+}1} \times S^{2k{+}1}$$
by starting with a degree two map on each sphere of the wedge $S^{2k{+}1} \vee S^{2k{+}1}$, and then extending by obstruction theory;
see \cite[XI: 1.16,  2.4]{whitehead-book}.

The ``propagation" method of Cappell, Davis, L\"offler and Weinberger (see
\cite{davis-loeffler1}, 
\cite[Theorem 1.6]{davis-weinberger2}) can now be used (in favourable circumstances) to construct free finite group actions on $\MK^{4k{+}2}$ from those on $S^{2k{+}1} \times S^{2k{+}1}$. 

\begin{theorem}\label{thm:odd order}
 Let $(S^{2k{+}1} \times S^{2k{+}1}, \pi)$ denote   a free, PL or smooth, orientation-preserving action of a finite odd order group $\pi$. Then 
\begin{enumerate}
\item In the PL case, there exists a free action $(\MK^{4k{+}2}, \pi)$ and a $\pi$-equivariant map
$f'\simeq f$ which is a $\pi$-equivariant degree four map.
 
\item In the smooth case, the $\pi$-action may be chosen to be smooth on some closed  manifold $N\cong_{PL} \MK^{4k{+}2}$
\end{enumerate}
\end{theorem} 
\begin{proof}
We first review the propagation method. Notice that the action of an odd order group induces the identity on homology.  The first step is to construct the homotopy pull-back diagram (where $q = |\pi|$ denotes the order of $\pi$):
$$\xymatrix{\ZZ  \ar[d] \ar[r] & Y{(\frac{1}{q})}\times K(\pi, 1)\ar[d]\cr
\XX{(q)} \ar[r]& Y{(0)}\times K(\pi, 1)}$$
where $\XX = (S^{2k{+}1} \times S^{2k{+}1})/\pi$ is the quotient space of the given free $\pi$-action, $Y= \MK^{4k{+}2}$, and $\XX(q)$,  $Y{(1/q)}$ and $Y(0)$ denotes Sullivan localizations of the spaces at $q$, $1/q$ or rationally (preserving the fundamental group information, as described in Taylor-Williams \cite[\S 1]{taylor-williams3}).
For the map $X(q) \to Y(0) \times K(\pi,1)$, note that 
$X(q)(\frac{1}{q}) \simeq \widetilde X(0)\times K(\pi,1)$,
and the degree four map $f\colon Y \to \widetilde X$ induces a
rational homotopy equivalence $Y(0) \simeq \widetilde X(0)$.

By Davis-L\"offler \cite[Lemma 1.4, Corollary 1.6]{davis-loeffler1}, we may assume that $\ZZ $ is a finite, oriented, simple Poincar\'e complex of dimension $4k{+}2$. In addition, we obtain a (simple) homotopy equivalence 
$$h\colon \MK^{4k{+}2} \to \widetilde \ZZ $$
to the universal covering of $\ZZ $. 
Since $\XX$ and $Y$ are both smooth or PL-manifolds, the local uniqueness of the Spivak normal fibration implies that there is a lifting
$$\xymatrix{&B\cat \ar[d]\cr \ZZ  \ar@{-->}[ur]\ar[r]& BG}$$
of the classifying map of the Spivak normal fibration $\nu_\ZZ $, with $CAT = DIFF$ or $CAT = PL$  depending on whether $\XX$ is smooth or just PL. This depends on the pull-back square, and the observation that $[Y(0),G/CAT] = 0$, since 
$$\pi_{r}(G/O) \otimes \bQ =\pi_{r}(G/PL) \otimes \bQ = 0,$$
for $r = 2k{+}1, 4k{+}2$. We now compare the surgery exact sequences
$$\xymatrix{0\ar[r]& \CSS(\ZZ ) \ar[d]^{tr} \ar[r]& \cN_\cat(\ZZ ) \ar[d]^{tr}\ar[r]& L^s_{4k{+}2}(\bZ\pi)\cong \cy 2 \oplus \widetilde L^s_{4k{+}2}(\bZ\pi)\ar[d]^{tr}\cr
0 \ar[r]& \CSS(\MK^{4k{+}2})\ar[r]& \cN_\cat(\MK^{4k{+}2}) \ar[r]& L^s_{4k{+}2}(\bZ) \cong \cy 2}$$
under the transfer induced by the universal covering $\widetilde \ZZ  \to \ZZ $ and the homotopy equivalence $h\colon \MK^{4k{+}2} \to \widetilde \ZZ $. We have substituted the well-known calculation $L^s_{4k+3}(\bZ \pi) = 0$ for $\pi$ of odd order 
\cite[\S 5.4]{wall-VI}, and claim that the structure set $\CSS(\ZZ ) \neq \emptyset$. 

The ordinary Arf invariant splits off $L^s_{4k{+}2}(\bZ) \cong \cy 2$, and the transfer map on $L$-groups is an isomorphism on this summand (since $\pi$ has odd order). The reduced $L$-group $\widetilde L^s_{4k{+}2}(\bZ\pi)$ is detected by the multi-signature invariant (see \cite[Prop.~12.1]{htaylor2}. 

In the PL case, we can choose a lifting of $\nu_\ZZ $ which agrees with the stable normal bundle of $\MK^{4k{+}2}$ under  the transfer, since
$$\cN_{PL}(\MK^{4k{+}2}) = [\MK^{4k{+}2}, G/PL] = \pi_{4k{+}2}(G/PL) = \cy 2,$$
and the only non-trivial normal invariant is mapped isomorphically
 to $L^s_{4k{+}2}(\bZ)\cong \cy 2$. In the smooth case, we can choose
  %only assume that there is a 
 any smooth normal invariant $\alpha \in \cN_{DIFF}(\ZZ )$
% in the kernel of the
% natural map $[Z, G/O] \to [Z, G/PL]$, 
such that the surgery obstruction of $tr(\alpha)$ is zero. In this case, the normal invariants
$$\cN_{DIFF}(\MK^{4k{+}2}) = [\MK^{4k{+}2}, G/O] = 
\pi_{2k{+}1}(G/O) \oplus  \pi_{2k{+}1}(G/O)  \oplus  \pi_{4k{+}2}(G/O)$$
are much more complicated, and any element $\beta =tr(\alpha) \in \cN_{DIFF}(\MK^{4k{+}2})$ with surgery obstruction zero will produce a possibly different smooth Kervaire manifold homotopy equivalent to $\MK^{4k{+}2}$. 
%The necessary condition on $\beta$ is just that it should  lie in the kernel of the natural map
%$[\MK^{4k{+}2}, G/O] \to [\MK^{4k{+}2}, G/PL]$.

Next we observe that if $\alpha \in \cN_\cat(\ZZ )$ is chosen so that $\beta=tr(\alpha)$ has trivial Arf invariant, 
then its surgery obstruction in $\widetilde L^s_{4k{+}2}(\bZ\pi)$ will be determined by the difference of multi-signatures
$$\sign_\pi(\NN ) - \sign_\pi(\ZZ )$$
in domain and range of a degree one normal map $\NN  \to \ZZ $ with normal invariant  $\alpha$ (see \cite[\S 13B]{wall-book}). Since $\NN $ is a closed PL or smooth manifold of dimension $4k{+}2$,  it has $\sign_\pi(\NN ) = 0$, and $\sign_\pi(\ZZ ) = 0$ since $\widetilde \ZZ  \simeq \MK^{4k{+}2}$. Therefore, there exists a smooth or PL-manifold $\NN  \simeq \ZZ $, whose universal covering $(\widetilde \NN , \pi)$ provides a free smooth or PL-action of $\pi$ on a Kervaire manifold $\MK^{4k{+}2}$. 
\end{proof}

\begin{remark} The roles of $\MK^{4k{+}2}$ and $S^{2k{+}1} \times S^{2k{+}1}$ can be reversed in this argument. This proves the other direction of Theorem C, so we conclude that the same odd order finite groups act freely on both manifolds.
\end{remark}

\section{Twisted doubles and the Spivak Normal Fibration}
\label{sec:spivak}

The main result of this section is a general result (see Proposition \ref{prop:Spivak}) about the Spivak normal fibration of a twisted double, or ``two patch space'' in the sense of Jones \cite{jones73}.  The statement is very natural, but we could not find it in the literature and so we give a proof. It will be used
in Section \ref{sec:smoothing_obstruction} for the proof of Theorem D.

Consider the following general situation: let $Q$ be
a compact, smooth oriented manifold with boundary $P$, and let $h \colon P \to P$ be an orientation-preserving homotopy equivalence which preserves the normal bundle of $P$: 
$h^*( \nu_P) \cong \nu_P$.  We form the Poincar\'e duality space
\[  \ZZ  : = Q \cup_h Q \]
 by gluing two copies of $Q$  together along $h$: this is a twisted double.
The Spivak normal fibration of $\ZZ $ may be identified with its classifying map,
\[  \nu_\ZZ  \colon \ZZ  \to BG, \]
and $\nu_\ZZ $ has a vector bundle reduction if and only if $B(i) \circ \nu_\ZZ  \colon \ZZ  \to BG \to B(G/O)$ is 
null-homotopic, where $B(i) \colon BG \to B(G/O)$ is the canonical map.
Since $B(G/O)$ is an infinite loop space \cite{boardman-vogt68}, it defines a generalised cohomology theory and we may 
consider the Mayer-Vietoris
sequence for $[\ZZ , B(G/O)]$ associated to the decomposition $\ZZ  = Q \cup_h Q$.  
The boundary map in this sequence is a homomorphism
\[ \delta_Z \colon [P, G/O] \to [Z, B(G/O)] .\]

\begin{proposition} \label{prop:Spivak}
%%%%%%%%
Let $\Nv(h) \in [P, G/O]$ be the normal invariant of $h \colon P \simeq P$.  Then
\[  [B(i) \circ \nu_\ZZ ] = \pm \delta_\ZZ(\Nv(h)) \in [\ZZ , B(G/O)] . \]
\end{proposition}

The proof of Proposition \ref{prop:Spivak} relies on foundational results about the Spivak normal fibrations of Poincar\'{e}
complexes which we now recall.
Let $(Y, \del Y)$ be an oriented Poincar\'{e} pair of formal dimension $m$ as defined in \cite{wall-pc1}.
The Spivak normal fibration of $Y$ is the unique
spherical fibration over $Y$ such that there is a homotopy class
\[ \rho_Y \in \pi_{m}(T(\nu_Y), T(\nu_{\del Y})) \]
such that $\rho_Y$ maps to the generator of $H_{m+k}(T(\nu_Y),T(\nu_{\del Y}); \bbZ) = \bbZ$
under the Hurewicz homomorphism (see \cite[Theorem A]{spivak67} and 
\cite[Theorem 3.2 and Corollary 3.4]{wall-pc1}). 
We call such a class $\rho_Y$ a {\em spherical reduction} for $\nu_Y$.
If
$ \del \colon \pi_{m+k}(T(\nu_Y), T(\nu_{\del Y})) \to \pi_{m+k-1}(T(\del Y)) $
denotes the boundary homomorphism, then $\del(\rho_Y)$ is a spherical reduction  for $\nu_{\del Y}$.
If $\XX$ is a closed manifold, then there is a canonical spherical reduction $\rho_\XX$ for $\nu_\XX$ 
obtained from embedding $\XX \subset S^{m+k}$.  
In general, a spherical reduction $\rho_Y$ is unique up to equivalence in the following sense.  
Let $\Aut(\nu_Y)$ be the group of homotopy classes of orientation-preserving
stable fibre homotopy equivalences of $\nu_Y$.

\begin{theorem}[{\cite[Theorem 3.5]{wall-pc1}}] \label{thm:Spivak_uniqueness}
%%%%%%%%
The mapping 
$$\Aut(\nu_Y) \to \pi_m(T(\nu_Y), T(\nu_{\del Y})), \quad e \mapsto e_*(\rho_Y), $$
defines a bijection between $\Aut(\nu_Y)$ and $\pi_{m+k}(T(\nu_Y), T(\nu_{\del Y}))_1$.
\end{theorem}

Theorem \ref{thm:Spivak_uniqueness} leads to an alternative definition of the normal invariant
of a tangential degree one normal map $(f, b) \colon \MM \to\XX$ of closed manifolds as we now explain.  
By Theorem \ref{thm:Spivak_uniqueness} there is the unique homotopy class of fibre homotopy equivalence
$e_b \in \Aut(\nu_\XX)$ such that 
\[ (e_{b})_*(\rho_{\XX}) = \pt([M, f, b]) .\]
%\in \pi_{m+k}(T(\nu_M))_1. \]
%
Moreover, if $\theta$ denotes the trivial stable spherical fibration, then by \cite[I.4.6]{browder72},
for any stable spherical fibration $\xi$ over a space $Y$ there is an isomorphism
\[ \gamma_\xi \colon \Aut(\theta) \to \Aut(\xi), \quad e \mapsto e + \id_{\xi} .\]
%
%with inverse
%
%\[ \gamma_\xi^{-1} \colon \Aut(\xi) \to \Aut(\theta), \quad b \mapsto b + \id_{-\xi} .  \]
%
We identify $\Aut(\theta) = [Y, SG]$ and define
\eqncount
\begin{equation} \label{eq:normal_invariant1}
\Nt([M, f, b]) = \gamma_{\nu_\XX}^{-1}(e_b) \in [\XX, SG].
\end{equation}

\begin{lemma} [See {\cite[(2.4)]{madsen-taylor-williams1}}] \label{lem:normal_invariants_agree}
%%%%%%%%
The normal invariant $\Nt([M, f, b])$ defined in \eqref{eq:normal_invariant1}
agrees with the normal invariant defined in 
\eqref{eq:normal_invariant}\,of Section \textup{\ref{sec:tangential_surgery}}.
\end{lemma}

\begin{proof}
%%%%%%%
Madsen, Taylor and Williams tell us \cite[p.\,450 above (2.4)]{madsen-taylor-williams1} that the lemma can be
directly checked using the definition of $S$-duality.  However, the authors refer to the book 
\cite{browder72} for the theory of Spivak fibrations, where only simply-connected Poincar\'{e} complexes are 
considered.  We therefore sketch the proof and verify that none of the relevant statements from \cite{browder72}
use the assumption of simple connectivity.

The proof of \cite[Corollary I.4.18]{browder72}, which is Browder's version of Theorem \ref{thm:Spivak_uniqueness},
contains two diagrams which may be joined together to give the following commutative diagram,
\[ \xymatrix{  \Aut(\varepsilon) \ar[d]^T  & \Aut(\theta) \ar[l]_{\gamma'} \ar[r]^{\gamma_{\nu_{\XX}}} & \Aut(\nu_\XX) \ar[d]^T \\
\{ T(\varepsilon), T(\varepsilon) \} \ar[d]^{\DD(\rho_{\XX})_* } \ar@{<->}[rr]^(0.465){\DD} & & \{T(\nu_\XX), T(\nu_\XX) \} \ar[d]^{\rho_{\XX}^*} \\
\{T(\varepsilon), S^0 \} \ar@{<->}[rr]^(0.475){\DD} & & \{S^m, T(\nu_\XX) \} }   \]
where $\epsilon = \nu_\XX \oplus (-\nu_\XX)$ is the trivial bundle, $\gamma'$ is an isomorphism defined analogously
to $\gamma_{\xi}$, $T$ denotes the induced map on the Thom space, $\DD$ denotes $S$-duality, 
$\DD(\rho_{\XX})_*$ and $\rho_{\XX}^*$ are induced by composition with the stable maps
$\rho_{\XX} \colon S^m \to T(\nu_\XX)$ and $\DD(\rho_{\XX}) \colon T(\varepsilon) \to S^0$.
The commutativity of the above diagram relies on \cite[Theorem I.4.16]{browder72} which makes no use of simple-connectivity.  

Note that taking adjoints gives an isomorphism $\Ad \colon \{ T(\varepsilon), S^0 \} \cong [\XX, SG]$
such that the composition $\Ad \circ \DD(\rho_{\XX})_* \circ T \circ \gamma' \colon \Aut(\theta) \to [\XX, SG]$ is the
canonical identification.  Note that $\rho_{\XX}^* \circ T(e_b) = \pt([M, f, b])$ and that $\DD(\pt([M, f, b]))$ is
the tangential normal invariant defined in \eqref{eq:normal_invariant}.  
On the other hand, $\gamma^{-1}(e_b)$ is the tangential normal invariant defined in \eqref{eq:normal_invariant1} 
and the commutativity of the diagram shows that the normal invariants agree.
\end{proof}

We now return to the general setting of Proposition \ref{prop:Spivak}, where
$Z : = Q \cup_h Q$ is a Poincar\'{e} complex obtained by gluing two copies of the smooth 
manifold $Q$ together along a homotopy equivalence $h \colon P \simeq P$,  such that
there is a bundle map $b \colon \nu_P \cong \nu_P$ covering $h$.  
Using a collar of $P \times [0, 1] \subset Q$ of the boundary $P \subset Q$, we regard $Z$ as the space
\[  Z = Q \cup_h (P \times [0, 1]) \cup_{\id_P} Q. \]
We define a stable vector bundle $\xi_{b}$ over the Poincar\'{e} complex 
$R : = Q \cup_{h} (P \times [0, 1])$,
\[ \xi_{b} : = \nu_Q \cup_{b} (\nu_P \times [0, 1]), \]
where we glue $P = \del Q$ to $P \times \{ 0 \} \subset P \times [0, 1]$: observe that 
$\xi_{b}|_{P \times \{ 1 \}} = \nu_P$.
Next recall that the fibre homotopy equivalence
$e_{b} \colon \nu_P \simeq \nu_P$ which is defined by the property that
\[ (e_{b})_*(\rho_P) = \mu_P([P, h, b]) = T(b)_*(\rho_P) \in \pi_{m+k}(T(\nu_P))_1 .\]

\begin{lemma} \label{lem:Spivak_of_Z}
%%%%%%%%
The spherical fibration $\xi : = \xi_{b} \cup_{e_{b}^{-1}} \nu_Q$ obtained by clutching the vector
bundles $\xi_{b}$ and $\nu_Q$ together along the fibre homotopy equivalence $e_b^{-1}$ is 
a model for the Spivak normal fibration of $Z$.
\end{lemma}

\begin{proof}
%%%%%%%
By \cite[Theorem 3.2 and Corollary 3.4]{wall-pc1}, it us enough to find a 
spherical reduction for $\xi$. 
We first identify a spherical reduction $\rho_R$ for $\xi_b = \nu_Q \cup_b(\nu_P \times [0,1])$ 
by gluing the spherical class $\rho_Q$ to the spherical class $T(b \times \id_{[0, 1]})_*(\rho_{P \times [0,1]})$.
Note that by construction $\del(\rho_R) = T(b)_*(\rho_P)$,
and by definition $(e_{b}^{-1})_*(T(b)_*(\rho_P)) = \rho_P$.  
Moreover, in the other copy of $Q$, we have $\del(\rho_Q) = \rho_P$ and 
thus, after choosing a homotopy between
representatives, we may form the homotopy class
\[  \rho_Z : = \rho_R \cup \rho_Q \in \pi_{m+k}(\xi). \]
Since the homotopy classes $\rho_R$ and $\rho_Q$ map to generators of 
$H_{m+k}(T(\nu_R), T(\nu_{P}); \bbZ)$ and
$H_{m+k}(T(\nu_Q), T(\nu_{P}); \bbZ)$ respectively, the Mayer-Vietoris sequence for
the decomposition $T(\xi) = T(\xi_{b}) \cup_{T(e_b^{-1})} T(\nu_Q)$ shows that $\rho_Z$ generates
$H_{m+k}(T(\xi); \bbZ)$.  Hence $\xi$ is a model for the Spivak normal fibration of $Z$.
\end{proof}

\begin{proof}[The proof of Proposition \ref{prop:Spivak}]
Let $\nu_Z \colon Z \to BSG$ also denote the classifying map of $\nu_Z$.  After the preparations above, it remains to identify
 the map $B(i) \circ \nu_Z \colon Z \to B(G/O)$ up to homotopy.
Since there is a fibration sequence
$$BO \longrightarrow BG \xra{B(i)} B(G/O), $$ the homotopy class of $B(i) \circ \nu_Z$ will not altered if we add a stable
vector bundle to $\nu_Z$.  For any stable vector bundle $\gamma$, let $-\gamma$ denote its inverse 
%so that $-\nu_M = \tau_M$ where $M$ is a manifold and $\tau_M$ is the stable normal bundle of $M$.  
and define the following stable vector bundle over $Z$:
\[  \Upsilon : = (-\xi_{b}) \cup_{\id_{(-\nu_P)}} (-\nu_Q). \]
The sum of spherical fibrations $\xi \oplus \Upsilon$ has a decomposition
\[ \xi \oplus \Upsilon = \bigl(\xi_b \oplus (-\xi_b) \bigr) \cup_{e_{b}^{-1} \oplus \id_{(-\nu_P)}} 
\bigl(\nu_Q \oplus (-\nu_Q)\bigr) \]
and is thus obtained by clutching two trivial bundles together along the fibre homotopy equivalence
\[
e : = (e_{b}^{-1} \oplus \id_{(-\nu_P)}) = \gamma^{-1}(e_{b}^{-1}) \in \Aut(\theta) \cong [P, SG] .
\]
It follows that there is an isomorphism of spherical fibrations 
\[  \xi \oplus \Upsilon \cong c_{\Sigma P}^*(\xi_e), \]
where $c_{\Sigma P} \colon Z \to \Sigma P$ is the map collapsing $Q \sqcup Q \subset Z$ to ${\rm pt} \sqcup {\rm pt}$
and $\xi_e$ is the spherical fibration over $\Sigma P$ obtained by clutching two copies of the trivial
spherical fibration over the cone of $P$ via $e$.

At this point we must briefly digress to discuss May's construction of $BH$, the classifying space of a topological monoid $H$ \cite[Proposition 8.7]{may-book}.  From this constuction we see that there is a canonical map $j^1_H \colon \Sigma H \to BH$ where $\Sigma H$ is the topological realisation of the $1$-simplex of the simplicial space used to define $BH$.  The map $j_H^1$ classifies the canonical principal $H$-fibration over $\Sigma H$ obtained by clutching two copies of the trivial $H$-fibration over the cone of $H$ via the identity map of $H$.

The isomorphism of spherical fibrations $\xi \oplus \Upsilon \cong c_{\Sigma P}^*(\xi_e)$ implies that 
the classifying map $\xi \oplus \Upsilon \colon Z \to BSG$ factors as
\[
 \xi \oplus \Upsilon \colon Z \xra{c_{\Sigma P}} \Sigma P \xra{ \Sigma(e)} \Sigma SG \xra{j_{SG}^1} BSG.   
 \]
It follows that $B(i) \circ \nu_Z = B(i) \circ (\xi \oplus \Upsilon)$ factors as
\[ 
B(i) \circ \nu_Z \colon Z \xra{c_{\Sigma P}} \Sigma P \xra{ \Sigma(i \circ e) } \Sigma (G/O )\xra{j_{G/O}^1} B(G/O). 
\]
Equivalently, $B(i) \circ \nu_Z = c_{\Sigma P}^*((j_{G/O}^1)_*(\Sigma(j \circ e))$.
Now $e = e_{b}^{-1} \oplus \id_{(-\nu_P)} = - \Nt(b)$ is the inverse of the tangential normal invariant of 
$(h, b) \colon P \simeq P$.  Hence 
$i \circ e  = - \eta(h) $ is the inverse of usual normal invariant of $h \colon P \simeq P$.
Finally, the composition 
\[ [P, G/O] \xra{~\Sigma~} [\Sigma P, \Sigma(G/O)] \xra{(j_{G/O}^1)^*} [\Sigma P, B(G/O)] \xra{c_{\Sigma P}^*} [Z, B(G/O)] \]
is, up to sign, the definition of the boundary map $\del_Z \colon [P, G/O] \to [\Sigma P, B(G/O)]$, and so
$[B(i) \circ \nu_Z] = \pm \del_Z(\eta(h))$.  This completes the proof of Proposition  \ref{prop:Spivak}.
\end{proof}

\section{The proof of Theorem D}
 \label{sec:smoothing_obstruction}
%%%%%%%%%%%%%%%%%%%%%%%%%%%%%%%%%%%%%%%%%%%%%%
In this section we return to the setting of Theorem \ref{thm:M}. 
Recall that $n=\dim \theta_j = 2^{j+1}-2$ and
that $\WW = D(\xi)$ is the disc bundle of an admisable bundle $\xi$.
In this section we suppose that $\xi$ is a \emph{vector bundle}.  
In Theorem \ref{thm:M}
we showed that the pinch map $\pa  \colon \VV  \to \VV$
of Definition \ref{def:p_alpha}
is homotopic to a PL-homeomorphism $ g(\alpha)\colon \VV \to \VV$, whenever $\alpha$ halves the Whitehead square. In other words, $x = \Sigma(\alpha)$ is an element of order two in $\theta_j$. 

 In Proposition \ref{prop:MK=tildeM} we showed that the PL-manifold
\[ \MM : = M(\xi, \alpha, g) =  \WW  \cup_{g(\alpha)} \WW \]
has universal cover PL-homeomorphic to $\MK$. Since $\xi$ is an admissible vector bundle, we have an action of \emph{linear type}.

Now let $\ZZ := W \cup_{p(\alpha)} W$ be the Poincar\'{e} complex underlying the PL-manifold
$M (\xi, \alpha, g)$ constructed in Proposition \ref{prop:MK=tildeM}. Let $\nu_Z$ denote the Spivak normal fibration
of $Z$ and let $\eta$ generate the stable $1$-stem.  

\begin{theorem} \label{thm:Spivak}
%%%%%%%%%%
Suppose that $w_2(\xi) = 0$. If $[\eta \cdot x_j] \neq 0 \in \Coker(J_{n+1}) = \pi_{n+1}(G/O)$, for some $x_j \in \theta_j$ with  $2x_j=0$,  then $\nu_\ZZ$ does not admit a vector bundle reduction. 
\end{theorem}
Before proving Theorem \ref{thm:Spivak} we verify that its hypotheses are satisfied.  By Corollary
\ref{cor:admisable_smooth_bundles} there are numerous admissible vector bundles $\xi$ over $\PP^{n+1}$ with
$w_2(\xi) = 0$; e.g.\,\,take $\xi = \nu_{\PP^{n+1}}$, the normal bundle of an embedding $\PP^{n+1} \to \bbR^{2n{+}2}$.
For the other hypothesis, we have

\begin{lemma} \label{lem:eta_theta_j}
%%%%%%%%
For $j = 3, 4, 5$ there exist $x_j \in \theta_j$ such that $[\eta \cdot x_j] \neq 0 \in \Coker(J_{2^{j+1}-1})$.
\end{lemma}

As a consequence of Theorem \ref{thm:Spivak} and Lemma \ref{lem:eta_theta_j}

\begin{corollary} \label{cor:not_smooth}
%%%%%%%%%
When $w_2(\xi) = 0$, and $x_j=\Sigma(\alpha_j)$ satisfies $[\eta \cdot x_j] \neq 0 \in \Coker(J_{2^{j+1}-1})$,
the PL-manifolds $\MM(\xi, \alpha_j, g_j)$ are not homotopy equivalent to smooth manifolds.
\end{corollary}

\begin{proof}[The proof of Lemma \ref{lem:eta_theta_j}]
%%%%%%
For $j = 3$, $\pi_{14}^S \cong \bbZ/2 \oplus \bbZ/2$ with generators $\sigma^2$ and $\kappa$ by \cite[p.\,189]{toda1}.
Since $\sigma^2$ is represented by $(S^7 \times S^7, f_7 \times f_7)$, 
where $f_7 \times f_7$ is the framing of
$S^7$ given by octonionic multiplication, we have $K(\sigma^2) = 1$.
Now \cite[p.\,189]{toda1} also shows that $[\eta \cdot \kappa] \neq 0 \in \Coker(J_{15})$, 
whereas, by \cite[p.\,257]{kochman-book2} $\eta \sigma^2 = 0$.
By \cite[Theorem 10.3]{toda1}, there is a homotopy class $\kappa_7 \in \pi_{21}(S^7)$ which stabilises
to $\kappa$.  
On the other hand by \cite{barratt-jones-mahowald3} the Kervaire invariant vanishes on the image of 
$\pi_{21}(S^7) \to \pi_{14}^S$ and hence $K(\kappa) = 0$.
Thus $x_3 := \kappa + \sigma \in \theta_3$ has $[\eta \cdot x_3] \neq 0 \in \Coker(J_{15})$.

For $j = 4, 5$ we assume that reader is familiar with using the mod $2$ Adams spectral sequence to compute
the $2$-primary part of $\pi_*^S$.  Recall that by \cite[Theorem 7.1]{browder1}, an element $x_j \in \pi_{2^{j+1}-2}$ 
has Kervaire invariant $1$ if and only if it represents $h_j^2$ in the Adams spectral sequence.
Now, for $j = 4, 5$, $h_1h_j^2$ is  a permanent cycle in the Adams spectral sequence with Adams filtration $3$:
see for example \cite[Theorem 8.3.2]{kochman-book2}.
%Since there are no elements of filtration two in dimension $2^{j+1}$, $h_1h_j^2$ represents a permanent cycle.  
Since multiplication by $h_1$
corresponds to multiplication by $\eta$ and since there are homotopy classes $x_4$ and $x_5$ representing 
$h_4^2$ and $h_5^2$, we may (ambiguously) 
denote such permanent cycles 
representing $h_1h_j^2$ by $\eta \cdot x_j$. 
Since the $2$-primary order of the image of $J_{2^{j+1}-2}$ is at least $2^5$, \cite[Theorem 1.6]{adams1965},.  
It follows that the element of order two in $\textup{Im}(J_{2^{j+1}-2})$ has Adams filtration
greater than $3$ and so $\eta \cdot x_j$ is not in the image of $J_{2^{j+1}-2}$.  In other words,
 $[\eta \cdot x_j] \neq 0 \in \Coker(J_{2^{j+1}-1})$.

For $j = 4$, the lemma also follows from \cite[Table A3.3]{ravenel-greenbook}: we take $x_4 = h_4^2$ and then 
$\eta \cdot x_4 = h_1h_4^2 \neq 0 \in \Coker(J_{31})$: here we use Tangora's names from 
\cite[Table A3.3]{ravenel-greenbook}. 
\end{proof}

We now turn to the proof of the remainder of Theorem \ref{thm:Spivak}. 
We first give an outline of the proof, reducing it to Proposition \ref{prop:Spivak} and a computational Lemma \ref{lem:sni} below.
We shall apply Proposition \ref{prop:Spivak} to the Poincar\'e complex $Z$ underlying $M$,
$$\ZZ = \WW  \cup_{\pa} \WW,$$ 
where for
$\VV = \bd\WW$, $\pa \colon V \simeq V$ is a tangential homotopy equivalence. 
%
%Since $\pa$ is homotopic to the PL-homeomorphism (by Theorem \ref{thm:M}), 
%the Poincar\'{e} complex $\ZZ$ is homotopy equivalent to $M$. 
  
Let $S^1 = \PP^1 \subset \PP^{n+1}$.
Since the bundle $\xi$ is orientable, its restriction to $S^{1} \subset \PP^{n+1}$ is trivial.  
Let $f_n \colon S^n \times S^1 \to \VV $ be the inclusion of this total space.
Since $\pa $ is the identity on $f_n(S^n \times S^{1}) \subset \VV $, there is a commutative diagram 
\eqncount
\begin{equation} \label{eq:MV}
\vcenter{ \xymatrix{ D^{n+1} \times S^1 \ar[d] & S^n \times S^1  \ar[d] \ar[l] \ar[r]  & D^{n+1} \times S^1 \ar[d] \\
\WW   & \VV  \ar[l] \ar[r] & \WW   } }
\end{equation}
which gives rise to an inclusion $f_{n+1} \colon S^{n+1} \times S^1 \to \ZZ$.    
From Proposition \ref{prop:Spivak} and diagram \eqref{eq:MV} we deduce that
\[ f_{n+1}^* ( [B(i) \circ  \nu_{\ZZ}] )= \delta_{S^{n+1} \times S^1}\bigl( f_n^*(\Nv(\pa)) \bigr) 
\in [S^{n+1} \times S^1, B(G/O)], \]
where $\delta_{S^{n+1} \times S^1} \colon [S^n \times S^1, G/O] \cong [S^{n+1} \times S^1, B(G/O)]$ is the boundary map in the
Mayer-Vietoris sequence for the decomposition 
$S^{n+1} \times S^1 = (D^{n+1} \times S^1) \cup_{\id} (D^{n+1} \times S^1)$.
Let 
$c_{S^{n+1}} \colon S^n \times S^1 \to S^{n+1}$ be the degree one collapse map.  Since the top cell stably splits off 
$S^{n} \times S^1$, the induced homomorphism 
$c_{S^{n+1}}^* \colon \pi_{n+1}(G/O) \to [S^n \times S^1, G/O]$
is a split injection.
Now since $\delta_{S^{n+1} \times S^1}$ is an isomorphism, to show that 
$$[B(i) \circ \nu_{\ZZ}] \neq 0 \in [\ZZ, B(G/O)],$$ it suffices to prove the following

\begin{lemma}  \label{lem:sni}
%%%%%%%%
$f_n^*(\Nv(\pa))= c_{S^{n+1}}^*([\eta \cdot x_j]) \in [S^n \times S^1, G/O]$.
\end{lemma}
We now prepare to give the proof of Lemma \ref{lem:sni}.
The statement amounts to showing that the diagram
\[ \xymatrix{S^n \times S^1\ar[r]^(0.6){f_n}\ar[d]_{c_{S^{n+1}}} 
& V \ar[d]^{\Nv(\pa)}\cr
S^{n+1} \ar[r]^{[\eta\cdot x_j]} & G/O} \]
commutes up to homotopy.
By Lemma \ref{lem:pa_is_tangential} there is a tangential normal map $(\VV , \pa , b)$ covering 
$\pa  \colon \VV  \to \VV$ and so by \eqref{eq:FSES}, $\Nv(\pa)$ factorises as 
\[ \Nv(\pa)  \colon \VV \xra{\Nt(b)} SG \xra{~i~} G/O, \]
where $\Nt(b)$ is the tangential normal invariant of $(\VV, \pa, b)$
and $i$ is the canonical map.
From Definition \ref{def:p_alpha} and the proof of Lemma \ref{lem:pinch}, we conclude that
$(\VV, \pa, b)$ is normally bordant to the disjoint union of tangential normal maps 
$(\VV, \Id, \Id) \sqcup (S^{2n+1}, x, b_x)$.
To describe the bundle map $(x, b_x) \colon \nu_{S^{2n+1}} \to \nu_\VV$,
we fix the notation $\zeta : = s^*(\nu_\VV)$.  Then $(b_x, x)$ factorises as in the following diagram,
\[ \xymatrix{  \nu_{S^{2n+1}} \ar[d] \ar[r]^{b_\alpha} & \pi^*(\zeta) \ar[d] \ar[r]^(0.55){b_{\pi}} & 
\zeta \ar[d] \ar[r]^(0.4){b_s} & T(\nu_V) \ar[d] \\
S^{2n+1}  \ar[r]^{\alpha} & S^{n+1} \ar[r]^(0.45){\pi} & \PP^{n+1} \ar[r]^(0.55){s} & V,   }   \]
%
%\[ \xymatrix{  S^{2n+1} \ar[d] \ar[r]^{\alpha} & S^{n+1} \ar[d] \ar[r] & \PP^{n+1} \ar[d] \ar[r] & V \ar[d] \\
%T(\nu_{S^{2n+1}}) \ar[r] & T(\pi^*(\zeta)) \ar[r] & T(\zeta) \ar[r] & T(\nu_V) }   \]
%
%Here $x = s \circ \pi \circ \alpha$ is the composition and we describe the bundle map $b_x$ next.
%
%\[  S^{2n+1} \xra{~\alpha~} S^{n+1} \xra{~\pi~} \PP^{n+1} \xra{~s~} \VV  \]
%
where $b_s \colon \zeta \to \nu_{\VV }$ is the canonical bundle map 
and $b_\pi \colon \pi^*(\zeta) \to \zeta$ and $b_\alpha \colon \nu_{S^{2n+1}} \to \pi^*(\zeta)$, 
are bundle maps covering $\pi$ and $\alpha$ respectively.   
We set $y := \pi \circ \alpha$ and $b_y : = b_\pi \circ b_\alpha$, 
and focus on the homotopy class 
\[ \rho_y := T(b_y)_*(\rho_{S^{2n+1}}) \in \pi_{2n+k+1}(T(\zeta)) \]
% which is related to the homotopy class
%$\rho_\alpha := \pt([S^{2n+1}, \alpha, b_\alpha])$ by the equation
%
%\[  \rho_y = T(b_\pi)_*(\rho_\alpha) \in \pi_{2n+k+1}(T(\zeta)) .\]
%
because $\Nt(b)$
%the tangential normal invariant $(\VV, \pa, b)$ 
is determined by $\rho_y$ according to Lemma \ref{lem:pinch}.
%given by
%
%\[  \Nt(b) =  [1] \ast (\DD(\rho_y) \circ s^!) \in [\VV, SG]. \]
%

Giving a precise description of $\rho_y$
% this stable homotopy class 
is a hard problem since the Thom space $T(\zeta)$ has many cells and
so we focus only on the top two cells of $\zeta$.  Let $\zeta_{n-1}$ be the restriction of $\zeta$ to 
$\PP^{n-1} \subset \PP^{n+1}$ and consider the map 
\[  c_{T(\zeta_{n-1})} \colon T(\zeta) \to T(\zeta)/T(\zeta_{n-1}) \simeq S^{n+k} \vee S^{n+k+1} \]
which collapses all but the top two cells of $T(\zeta)$.  The following key computational lemma 
is a consequence of the assumption $w_2(\xi) = 0$ in Theorem \ref{thm:Spivak}.  
We defer its proof until after the proof of Lemma \ref{lem:sni}.

\begin{lemma} \label{lem:rho_y}
%%%%%%%
$(c_{T(\zeta_{n-1})})_*(\rho_y) = (\eta \cdot x_j, 0) \in \pi_{2n+k+1}(T(\zeta)/T(\zeta_{n-1})) \cong \pi_{n+1}^S \oplus \pi_n^S$.
\end{lemma}

\begin{proof}[The proof of  Lemma \ref{lem:sni}]  
%%%%%%%
%By Lemma \ref{lem:pinch}, the normal invariant of $\pa \colon V \simeq V$ is given by the composition 
%
%\[ \Nv(\pa) \colon \VV  \xra{~t^!~} T(\nu_s) \xra{~D(\rho_y)~} QS^0_0 \xra{~[1]\ast~} SG \xra{~~i~~} G/O, \]
%
By Lemma \ref{lem:pinch} and the construction of $f_n \colon S^n \times S^1 \to \VV$,
there is a commutative diagram
\begin{equation} \label{eq:f_n}
 \xymatrix{  S^n \times S^1 \ar[d]^(0.45){c_{S^n \vee S^{n+1}}} \ar[rr]^(0.55){f_n} & & \VV \ar[d]^(0.45){s^!} \ar[rr]^{\Nt(b)} 
& & SG \ar[d]^(0.45)= \\
S^n \vee S^{n+1} \ar[rr]^(0.575){i_{T(\nu_s)}} & & T(\nu_s) \ar[rr]^{\DD(\rho_y)} & & SG \ar[rr]^i && G/O ,   }  
\end{equation}
where $c_{S^n \vee S^{n+1}}$
% \colon S^n \times S^1 \to S^n \vee S^{n+1}$ 
is the map collapsing $S^1$, $i_{T(\nu_s)}$
% \colon S^n \vee S^{n+1} \to T(\nu_s)$ be 
is the inclusion of the bottom two cells of the Thom space $T(\nu_s)$
and we recall that $\DD(\rho_y)$ is the adjoint of the $\sw$-dual of $\rho_y$ defined as in
Lemma \ref{lem:normal_invariant_1}.

Since $c_{S^{n+1}} \colon S^n \times S^1 \to S^{n+1}$ factors over $c_{S^n \vee S^{n+1}}$ is the obvious way,
to prove Lemma \ref{lem:sni},  it will be enough to understand the map 
$\DD(\rho_y) \circ i_{T(\nu_s)} \colon S^n \vee S^{n+1} \to SG$.
Since $\zeta = s^*(\nu_\VV)$, the $\sw$-dual of $T(\nu_s)$ is $T(\zeta)$ by Lemma \ref{lem:Tnut} \eqref{it:SWT(b_t)}.
In particular the $\sw$-dual of $i_{T(\nu_s)}$ is $c_{T(\zeta_{n-1})}$ and there is a commutative diagram
with rows of stable maps related by $\sw$-duality: 
\[  \xymatrix{  S^n \vee S^{n+1} \ar[d]^(0.45){\SW}  & & T(\zeta) \ar[d]^(0.45){\SW} \ar[ll]_(0.4){c_{T(\zeta_{n-1})}}
 & & S^{2n+1} \ar[d]^(0.45){\SW} \ar[ll]_(0.45){{\rho_y}} \\
S^{n+1} \vee S^n \ar[rr]^(0.55){i_{T(\nu_s)}} & & T(\nu_s) \ar[rr]^(0.5){\SW(\rho_y)} & & S^0 . } \] 
%

%%
%\[  \xymatrix{  S^{2n+1} \ar[d]^(0.45){\SW} \ar[rr]^(0.5){\rho_y}  & & T(\zeta) \ar[d]^(0.45){\SW} 
%\ar[rr]^(0.45){c_{T(\zeta_{n-1})}}  & & S^n \vee S^{n+1} \ar[d]^(0.45){\SW} \\
%S^0 & & T(\nu_s) \ar[ll]_{\SW(\rho_y)} & & S^{n+1} \vee S^n \ar[ll]_(0.5){i_{T(\nu_s)}}. } \] %\SW(c_{T(\zeta_{n-1})})
%
%Let $i_n \colon S^n \to S^n \vee S^{n+1}$ be the inclusion and let $j_{n+1} \colon S^{n+1} \vee S^n \to S^{n+1}$ be
%the map collapsing $S^n$ to a point.  
By Lemma \ref{lem:rho_y}, the composition $c_{T(\zeta_{n-1})} \circ \rho_y$ ignores the $S^{n+1}$ factor
of the target wedge and maps to $S^n$ via $\eta \circ x_j$.  It follows that 
$\SW(\rho_y) \circ i_{T(\nu_s)}$ is given via projecting to $S^{n+1}$ and mapping with $\eta \cdot x_j$.
Passing to the adjoint of $\SW(\rho_y)$, $\DD(\rho_y)$,
it follows that $i_{T(\nu_s)} \circ \DD(\rho_y)$ is null homotopic when restricted to $S^n$, and 
represents the homotopy class $\eta \cdot x_j \in \pi_{n+1}(QS^0_0) = \pi_{n+1}^S$
on $S^{n+1}$.  The maps $[1]\ast$ and $i$ carry this homomtopy class to the element $[\eta \cdot x_j] \in \pi_{n+1}(G/O) = \Coker(J_{n+1})$.  
The fact that $\eta(\pa) = i \circ \Nt(p)$ and the commutative diagram \eqref{eq:f_n} now give the proof 
of Lemma \ref{lem:sni}.
\end{proof}

Next we turn to the proof of Lemma \ref{lem:rho_y}.
Let us first establish some basic facts about the stable bundle $\zeta$.
Recall from the proof of Lemma \ref{lem:pa_is_tangential}, that there is a bundle isomorphism
%
%We first describe the bundle $\zeta$ in more detail and compute its first and second Stiefel-Whitney classes.
%Let us first understand the bundle $s^*(\nu_{\VV })$.  
%Since $\VV  = S(\xi)$ is the total space of the sphere bundle of $\xi$ with
%bundle projection $\pi_\xi \colon \VV  \to \PP^{n+1}$
%we have
%
%\[ \nu_{\VV } = \pi_\xi^*(\nu_{\PP^{n+1}}) \oplus \pi_\xi^*(-\gamma) \]
%
%where $\gamma$ is the stable bundle defined by $\xi$ and $-\gamma$ its stable inverse.   Since $s \circ \pi = \Id_{\PP^{n+1}}$, 
%
\[  \zeta = s^*(\nu_{\VV}) \cong \nu_{\PP^{n+1}} \oplus (-\gamma), \]
where $\gamma$ is the stable bundle defined by $\xi$.
%
%Now, the stable normal bundle of $\PP^{n+1}$ is know to be $-(n{+}2)\eta$ where $\eta$ is the canonical line bundle over $\PP^{\infty}$.
%For later use we record the following computation of the second Stiefel-Whitney class of $s^*(\nu_{\VV })$:
\begin{lemma} \label{lem:w2}
$w_1(\zeta) = w_2(\zeta) = 0$.
\end{lemma}

\begin{proof}
%%%%%%%
Since $n{+}2$ is a power of two and $\nu_{\PP^{n+1}} = -(n{+}2) \cdot \eta$, 
%\cite[Theorem 4.5]{milnor-stasheff1}, 
we have the equality $w_1(\nu_{\PP^{n+1}}) = w_2(\nu_{\PP^{n+1}}) = 0$.  
Recall that $\pi_\xi \colon \VV \to \PP^{n+1}$ is the bundle projection.
Since $\VV $ is orientable and $\nu_{\VV } = \pi_\xi^*(\nu_{\PP^{n+1}}) \oplus \pi_\xi^*(-\gamma)$, 
it follows that $w_1(-\gamma) = 0$ and so $w_2(-\gamma) = w_2(\gamma) = 0$, where the last equality holds by assumption.
The Cartan formula now gives $w_1(\zeta) = w_1(\nu_{\PP^{n+1}}) + w_1(-\gamma) = 0$ and so
$ w_2(\zeta) = w_2\bigl( \nu_{\PP^{n+1}} \oplus (-\gamma) \bigr) = 0.$
%w_2(\nu_{\PP^{n+1}}) + w_1(\nu_{\PP^{n+1}}) \cdot w_1(-\gamma) + w_2(-\gamma)  = 0 .\]
%
\end{proof}

Since $\zeta$ is a stable real vector bundle over $\PP^{n+1}$ it has an extension $\widehat \zeta$ to $\PP^{n{+}2}$,
and there is a homotopy equivalence
\[  T(\wh \zeta) \simeq T(\zeta) \cup_\phi e^{n+k+2}, \]
where $\phi \colon S^{n+k+1} \to T(\zeta)$ is the attaching map of the top cell of $T(\wh \zeta)$.  
We shall establish
two important facts about the homotopy class of $\phi$ in Lemma \ref{lem:T(zeta)} below.  
%Recall that $\zeta_{n-1} = \zeta|_{\PP^{n-1}}$, let
%
%\[  c_{T(\zeta_{n-1})} \colon T(\zeta) \to T(\zeta)/T(\zeta_{n-1}) \simeq S^{n+k} \vee S^{n+k+1} \]
%
%be the collapse map
Let 
\[ c^0 \colon T(\zeta) \to T(\zeta)/S^k \]
be the map collapsing the Thom cell of $T(\zeta)$ to a point.
%Since $\pi^*(\gamma)$ is a trivial stable bundle over $S^{n+1}$, the same is true of $\pi^*(-\gamma)$ and of course 
%$\pi^*\nu_{\PP^{n+1}}$ is trivial.  
In the proof of Lemma \ref{lem:pa_is_tangential} we proved that $\pi^*(\zeta)$ is trivial.
%a trivial stable bundle over $S^{n+1}$.  
Hence there is a homotopy equivalence 
$T(\pi^*(\zeta)) \simeq S^k \vee S^{n+k+1}$ and the bundle map $b_\pi \colon \pi^*(\zeta) \to \zeta$ induces a map
\[  T(b_\pi)/S^k \colon S^{n+k} \to T(\zeta)/S^k .\]

\begin{lemma} \label{lem:T(zeta)}
%%%%%%%
The homotopy class $[\phi] \in \pi_{n+k+1}(T(\zeta))$ satisfies:
\begin{enumerate}
\item \label{it:T(zeta)1}
$(c^{0})_*(\phi) = [T(b_\pi)/S^k] \in \pi_{n+k+1}(T(\zeta)/S^k)$,
\item \label{it:T(zeta)2}
$(c_{T(\zeta_{n-1})})_*(\phi) = (\eta, 2) \in \pi_{n+k+1}(S^{n+k} \vee S^{n+k+1}) \cong \pi_1^S \oplus \pi_0^S$.
\end{enumerate}
\end{lemma}

\begin{proof}
%\eqref{it:T(zeta)1}  
(i) Let $\pi_{n{+}2} \colon \PP^{n{+}2} \to S^{n{+}2}$ be the covering projection so that 
$\pi_{n{+}2}|_{\PP^{n+1}} = \pi \colon \PP^{n+1} \to S^{n+1}$.  The bundle maps $b_\pi$ and $b_{\pi_{n{+}2}}$ covering
$\pi$ and $\pi_{n{+}2}$ induce a commutative diagram of map of Thom spaces with Thom cells collapsed:
\[  \xymatrix{  T(\pi^*(\zeta))/S^k \ar[d] \ar[rr]^{T(b_\pi)/S^k} & & T(\wh \zeta)/S^k \ar[d] \\
(T(\pi^*(\zeta))/S^k) \cup (D^{n+k+2} \sqcup D^{n+k+2}) \ar[rr]^(0.57){T(b_{\pi_{n{+}2}})/S^k} & &
(T(\zeta)/S^k) \cup_{c^0 \circ \phi} D^{n+k+2}   }  \]
But the inclusion $T(\pi^*(\zeta)) \to T(\pi^*(\wh \zeta))$ is homeomorphic to the standard inclusion of 
a hypersphere $S^{n+k+1} \to S^{n+k+2}$
and $T(b_{\pi_{n{+}2}})$ maps the interior of each $D^{n+k+2}$ homeomorphically onto the interior of the 
single $D^{n+k+2}$ in its target.  Hence $T(b_\pi)/S^k$ is homotopic to $c^0 \circ \phi$.

%\eqref{it:T(zeta)2}
(ii) The space $T(\wh \zeta)/T(\zeta_{n-1})$ is homotopy equivalent to a $3$ cell complex 
and so there is homotopy equivalence
\[ T(\wh \zeta)/T(\zeta_{n-1}) \simeq (S^{n+k} \vee S^{n+k+1}) \cup_{c_{T(\zeta_{n-1})} \circ \phi} D^{n+k+2}, \]
where we have use the homotopy equivalence $T(\zeta)/T(\zeta_{n-1})  \simeq S^{n+k} \vee S^{n+k+1}$.
If we define $j_{n+k+1} \colon S^{n+k} \vee S^{n+k+1} \to S^{n+k+1}$ to be the map collapsing $S^{n+k}$ to a point,
then the degree of $j_{n+k+1} \circ c_{T(\zeta_{n-1})} \circ \phi$ is determined by the homology group 
$H_{n+k+2}(T(\wh \zeta))$ which is isomorphic to $\bbZ/2$ since $\wh \zeta$ is non-orientable.
Choosing orientations appropriately, we have determined by the second component of $(c_{T(\zeta_{n-1})})_*([\phi])$.

We can read off the homotopy class of the second component fo $c_{T(\zeta_{n-1})} \circ \phi$ from the action of 
$Sq^2$ in $T(\wh \zeta)/T(\zeta_{n-1})$ since $Sq^2$ detects $\pi_1^S$.  

The collapse map 
$\wh c_{T(\zeta_{n-1})} \colon T(\wh \zeta) \to T(\wh \zeta)/T(\zeta_{n-1})$ induces an isomorphism on mod~2 cohomology 
in dimensions $n+k$ and $n+k+2$ and hence we can work in $H^*(T(\wh \zeta); \bbZ/2)$.
Let $x \in H^1(\PP^{n+1}; \bbZ/2)$ be a generator and let $U$ be the Thom class of $\wh \zeta$.  Then
$H^{n+k}(T(\wh \zeta); \bbZ/2)$ is generated by $x^n U$ and we compute 
%
%\[ Sq^1(x^{n+1}U) = x^{n{+}2}U \quad \text{ and } \quad Sq^2(x^n U) = x^{n{+}2} U, \]
\[ \quad Sq^2(x^n U) = x^{n{+}2} U, \]
since $n = 2^{j+1}-2$, $Sq^i(U) = w_i(\wh \zeta)U$ and $w_i(\wh \zeta) = w_i(\zeta) = 0$ for $i = 1, 2$.
This shows that $Sq^2$ maps non-trivially to the top cell of $T(\wh \zeta)/T(\zeta_{n-1})$ and it follows
that the second component of $(c_{T(\zeta_{n-1})})_*(\phi)$ is $\eta$.
\end{proof}

%As a direct consequence of Lemma \ref{lem:T(zeta)} we can now say the following about the homotopy class 
%$\rho_y \in \pi_{2n+k+1}(T(\zeta))$.

%\remdc{fix up here}

%\begin{lemma} \label{lem:rho_y}
%%%%%%%
%$(c_{T(\zeta_{n-1})})_*(\rho_y) = (\eta \cdot x_j, 0) \in \pi_{2n+k+1}(T(\zeta)/T(\zeta_{n-1})) \cong \pi_{n+1}^S \oplus \pi_n^S$.
%
%\end{lemma}

\begin{proof}[Proof of Lemma \ref{lem:rho_y}]
The map $\alpha \colon S^{2n+1} \to S^{n+1}$ stabilises to $x_j$ and  since $\pi^*(\zeta)$ is trivial, the induced map on
Thom complexes with Thom cells collapsed, 
\[ T(b_\alpha)/S^k \colon T(\alpha^*\pi^*(\zeta))/S^k \to T(\pi^*(\zeta))/S^k, \]
is identified with the $k$-fold suspension of $\alpha$.  But this map can
also be identified with the map $\rho_\alpha =T(b_\alpha)_*(\rho_{S^{2n+1}})$ composed with the collapse
of the Thom cell of $T(\pi^*(\zeta))$.   
We recall that $\rho_y = T(b_\pi)_*(\rho_\alpha)$ and
let $c_{T(\zeta_{n-1})}^0 \colon T(\zeta)/S^k \to T(\zeta)/T(\zeta_{n-1})$ denote
the collapse map, so that
$c_{T(\zeta_{n-1})} = c_{T(\zeta_{n-1})}^0  \circ c^0$.
Then applying Lemma \ref{lem:T(zeta)} we have
\begin{eqnarray*}  &(c_{T(\zeta_{n-1})})_*(\rho_y) &= 
(c^0_{T(\zeta_{n-1})})_*\bigl((c^0)_*(T(b_\pi)_*(\rho_\alpha)) \bigr) = 
(c^0_{T(\zeta_{n-1})})_*\bigl( (T(b_\pi)/S^k)_*(x_j) \bigr) \\ &&= 
(c_{T(\zeta_{n-1})})_*(\phi \circ x_j) = (\eta \cdot x_j, 0) .
\end{eqnarray*}
\end{proof}

%We conclude the paper with a brief discussion of the smoothability of the manifolds
%$M(\xi, \alpha, g)$.

\begin{remark}
Our assumption in Theorem D that $\WW  \to \PP^{n+1}$ is a smooth fibre bundle ensures that
$\WW$ is a smooth manifold with $\bd \WW = \VV$.  Hence $\MM$ is the twisted double
of a smooth manifold along a PL-homeomorphism, but is not smoothable.
Since $\wt \MM \cong_{PL} \MK^{2n{+}2}$, it is interesting to ask whether $\MM$ admits a smooth structure over some skeleton.

\begin{lemma} \label{lem:(n{+}1)_smooth}
%%%%%%%
The PL-manifold $M$ admits a smooth structure over its $(n{+}1)$-skeleton.
\end{lemma}
\begin{proof}
%%%%%%%
%
Let us denote the copies of $\WW$ used to build $M$ by $\WW_0$ and $\WW_1$.
If we collapse $\WW_0$ to a point then we obtain $\WW_1/\del \WW_1$, the Thom space of $\xi$.
Since $\xi$ has rank $(n{+}1)$, $T(\xi)$ has a $CW$-decomposition starting
from $S^{n+1}$ and attaching cells of dimension $(n{+}2)$ and higher.  
It follows that $M$ has a $CW$-decomposition with $(n{+}1)$-skeleton $\PP^{n+1} \vee S^{n+1}$
where $\WW_0$ thickens $\PP^{n+1}$.  Up to homotopy, the remaining $S^{n+1}$ is represented
by the union of the fibre discs in $\WW_0 \to \PP^{n+1}$ and $\WW_1 \to \PP^{n+1}$.
Let $D^{n+1}_1 \subset \WW_1$ be such a fibre and let $D^{n+1} \times D^{n+1}_1 \subset \WW_1$ be a tubular
neighbourhood of $D^{n+1}_1$ which meets $\del \WW_1$ at $D^{n+1} \times S^n_1$.
It is enough to show that the PL-manifold 
\[  \WW_2 : = \WW_0 \cup_{g^{-1}|_{D^{n+1} \times S^n_1}} (D^{n+1} \times D^{n+1}_1) \]
admits a smooth structure.  
By \cite[Theorem 5.3]{hirsch-mazur-book}, the obstruction to extending the smooth structure on $\WW_0$
to $\WW_2$ is an obstruction class
\[ \omega \in H^{n+1}(\WW_2, \WW_0; \pi_n(PL/O)) \cong \bbZ. \]
This obstruction is natural for coverings and $\omega$ pulls back to the obstruction class
$\wt \omega \in H^{n+1}(\wt \WW_2, \wt \WW_0; \pi_n(PL/O)) \cong \bbZ^2$
which we may identify as $\wt \omega = (\omega, \omega)$. 
Now $\wt M \cong \MK$ and $\wt \WW_2 \subset \MK$ is homotopy equivalent to a wedge of 
three $(n{+}1)$-spheres.  Since $\MK$ is smoothable away from a point, it follows that
$\wt \omega = 0$ and hence that $\omega = 0$.
\end{proof}

\end{remark}
%%%%%%%%%%%%%%%%%%%%%%%%%%%%%%%%%%%%%%%%%%%%%%

%%%%%%%%%%%%%%%%%%%%%%%%%%%%%%%%%%%%%%%%%%%%%%%%%
% 
%\bibliographystyle{ih}
%\bibliography{ihmain,dc}
%\end{document}
%
%%%%%%%%%%%%%%%%%%
\providecommand{\bysame}{\leavevmode\hbox to3em{\hrulefill}\thinspace}
\providecommand{\MR}{\relax\ifhmode\unskip\space\fi MR }
\providecommand{\MRhref}[2]{%
  \href{http://www.ams.org/mathscinet-getitem?mr=#1}{#2}
}
\providecommand{\href}[2]{#2}

\end{document}